\documentclass[11pt]{article}
\usepackage{amsmath}
\usepackage{amsthm}
\usepackage{amsfonts}
\usepackage{amssymb}
\usepackage{bm}
\usepackage{microtype}
\usepackage[colorlinks=true]{hyperref}
\theoremstyle{plain}
\newtheorem{thm}{Theorem}[section]
\newtheorem{lem}[thm]{Lemma}

\newtheorem{assumption}[thm]{Assumption}

\newtheorem{definition}[thm]{Definition}
\theoremstyle{remark}
\newtheorem{rmk}[thm]{Remark}
\usepackage{pifont}

\usepackage{xcolor}
\usepackage{graphicx}
\usepackage{caption}
\usepackage{subcaption} 
\usepackage{float} 
\usepackage{booktabs}
\usepackage{tabularx}
\usepackage{array}
\title{Mean field and N-agent games for optimal relative consumption-investment with jump risk and common noise}
\author{Yiming Jiang, Fuxing Li, Yawei Wei, Zimeng Zheng}
\begin{document}
	\maketitle
	{
		
	}
	\begin{abstract}
		This paper studies an optimal consumption--investment problem for competitive agents in an \(N\)-player game and its associated mean field game. Each agent invests in an individual risky asset subject to idiosyncratic noise, common noise and downward jump risk, and the interaction among agents is induced by relative performance concerns in both consumption and terminal wealth. In the mean field limit, we characterize a deterministic mean field equilibrium in analytical form by using the stochastic maximum principle. Numerical experiments are presented to illustrate the resulting equilibrium and its financial implications. Finally, based on the obtained mean field equilibrium, we construct an approximate Nash equilibrium for the \(N\)-player game. This model is motivated by \cite{Merton1971} and \cite{Lacker2020}.
	\end{abstract}
	\section{Introduction and main results}
	We consider a large financial market populated by $N$ agents sharing a common finite time horizon $[0,T]$, where $T>0$. Each agent $i$ may invest in a common riskless bond with zero interest rate or in an individual stock $i$. For $i=1,\ldots,N$, the price process of stock $i$ is given by the following stochastic differential equation (SDE):
	\begin{equation}\label{Ssde}
		\frac{dS_t^i}{S_{t-}^i}=b_i dt+\sigma_idW_t^i+\sigma_i^0dW_t^0-dM_t^i,\qquad t\in[0,T]
	\end{equation}
	with constant parameters $b_i>0$ and $\sigma_i,\sigma_i^0>0$. Here, $W^0=(W_t^0)_{t\in[0,T]}$ appears in the price dynamics of all agents and represents the common noise in the financial market, while $W^i=(W_t^i)_{t\in[0,T]}$ denotes the idiosyncratic noise associated with agent $i$. The Brownian motions $W^0,W^1,\ldots,W^N$ are independent on a filtered probability space $(\Omega,\mathcal F,\mathbb F,\mathbb P)$, where
	$\mathbb F=(\mathcal F_t)_{t\in[0,T]}$ satisfies the usual conditions. Moreover, we set the filtration $\mathbb{F}^0=\left( \mathcal{F}^0_t \right)_{t\in [0,T]}=\left( \sigma\left( W_s^0;s\le t \right) \right)_{t\in [0,T]}$. In the above equation, $M^i=(M_t^i)_{t\in[0,T]}$ denotes the compensated process of $N^i=(N_t^i)_{t\in[0,T]}$, where $N^i$ is a Poisson process with intensity $\lambda_i$.
	The presence of $M^i$ allows the stock price to be subject to downward jump risk, a classical modeling feature already considered by Merton \cite{Merton1976}. In particular, the Poisson processes $N^1,\ldots,N^N$ are mutually independent and independent of the Brownian motions $W^0,W^1,\ldots,W^N$. We define the global market filtration
	$\mathbb G=(\mathcal G_t)_{t\in[0,T]}=(\mathcal F_t\vee \sigma((N^1_s,\ldots,N^N_s)^\top;\,0\le s\le t))_{t\in[0,T]}$
	as the right-continuous augmentation by $\mathbb P$-null sets. Then, for each
	$i=1,\ldots,N$, the compensated process $M_t^i=N_t^i-\lambda_i t$ is a
	$(\mathbb P,\mathbb G)$-martingale on $[0,T]$. In addition, the Brownian motions
	$W^0,W^1,\ldots,W^N$ remain Brownian motions with respect to $\mathbb G$.
	A similar filtration setting can be found in Bo et al. \cite{Bo2024a}.
	
	For $i=1,\ldots,N$, let $\pi^i=(\pi_t^i)_{t\in[0,T]}$ denote the proportion of wealth that agent $i$ invests in stock $i$, and let $c^i=(c_t^i)_{t\in[0,T]}$ denote the instantaneous consumption rate per unit wealth of agent $i$. Then the corresponding self-financing wealth processes $X^i=(X_t^i)_{t\in[0,T]}$ satisfy the following SDE:
	\begin{equation}\label{SDEi}
		\left\{
		\begin{aligned}
			&dX_t^i=\left(\pi_t^ib_i-c_t^i\right)X_t^i dt+\pi_t^i X_t^i\sigma_idW_t^i+\pi_t^i X_t^i\sigma_i^0dW_t^0-\pi_t^i X_{t-}^idM_t^i, \\
			&X_0^i=x_0^i,
		\end{aligned}
		\right.
	\end{equation}
	where $x_0^i>0$ denotes the initial wealth of agent $i$. We now introduce the admissible control set for each agent.
	\begin{definition}[Admissible control set for agent $i$]\label{mathcalAi}
		For each agent $i=1,\ldots,N$, the admissible control set $\mathcal A_i$ consists of all pairs of processes
		$(\pi^i,c^i)=(\pi_t^i,c_t^i)_{t\in[0,T]}$ such that $\pi^i$ is $\mathbb G$-predictable, $c^i$ is $\mathbb G$-progressively measurable,
		\[
		(\pi_t^i,c_t^i)\in [D_0,1-\epsilon_0]\times (0,\infty)
		\quad \text{for all } t\in[0,T],\quad \mathbb P\text{-a.s.},
		\]
		and
		\[
		\int_0^T c_t^i\,dt<\infty,\qquad \mathbb P\text{-a.s.}
		\]
		Here $D_0\in\mathbb R$ is a constant and $\epsilon_0\in(0,1)$ is a sufficiently small constant.
	\end{definition}
	\begin{rmk}\label{rmk:admissible}
		The admissibility conditions are imposed to ensure that the wealth equation
		\eqref{SDEi} is well-posed and that the corresponding wealth process remains strictly positive. In particular, the lower bound $D_0$ rules out unbounded short positions, while the constraint $\pi_t^i\le 1-\epsilon_0$ prevents bankruptcy after a downward jump, since $X_t^i=X_{t-}^i(1-\pi_t^i)>0$.
	\end{rmk}
	
	The aim of agent $i$ is to maximize the following objective functional over all admissible strategies $(\pi^i,c^i)\in\mathcal A_i$:
	\begin{equation}\label{object_i}
		J_i((\pi^i,c^i)_{i=1}^N)=\mathbb{E}\left[\int_0^T U(c_t^iX_t^i(\overline{c_tX_t})^{-\theta_i};\gamma_i)dt+\varepsilon_iU(X_T^i\overline{X_T}^{-\theta_i};\gamma_i)\right],
	\end{equation}
	where $U(x;\gamma)=\frac{1}{\gamma}x^\gamma$ is a constant relative risk aversion (CRRA) utility function. Moreover, $\overline{c_tX_t}:=\left(c_t^1X_t^1\cdots c_t^N X_t^N\right)^{1/N}$ and $\overline{X_T}:=\left(X_T^1\cdots X_T^N\right)^{1/N}$ denote the population geometric averages of consumption and terminal wealth, respectively. The parameters $\gamma_i\in(0,1)$ and $\theta_i\in[0,1]$ characterize the risk preference and the competition weight of agent $i$, while $\varepsilon_i>0$ measures the relative importance assigned to terminal wealth.
	
	In the presence of jump risk in the wealth dynamics \eqref{SDEi}, the explicit characterization of Nash equilibrium strategies for the $N$-agent game is generally intractable. Nevertheless, we can construct an approximate Nash equilibrium for the $N$-agent system. The definition is given as follows.
	\begin{definition}[Approximate Nash equilibrium]
		An admissible strategy profile
		\[
		(\bm\pi^*,\bm c^*)
		=
		\left((\pi^{*,1},c^{*,1}),\ldots,(\pi^{*,N},c^{*,N})\right)
		\in \mathcal A_1\times\cdots\times\mathcal A_N
		\]
		is called a \(u_N\)-Nash equilibrium if there exists a sequence
		\(u_N\to0\) as \(N\to\infty\) such that, for every \(i=1,\ldots,N\),
		\[
		\sup_{(\pi^i,c^i)\in\mathcal A_i}
		J_i\left((\pi^i,c^i),(\bm\pi^*,\bm c^*)^{-i}\right)
		\leqslant
		J_i\left((\pi^{*,i},c^{*,i}),(\bm\pi^*,\bm c^*)^{-i}\right)
		+u_N .
		\]
		Here,
		\[
		(\bm\pi^*,\bm c^*)^{-i}
		:=
		\left(
		(\pi^{*,1},c^{*,1}),\ldots,
		(\pi^{*,i-1},c^{*,i-1}),
		(\pi^{*,i+1},c^{*,i+1}),\ldots,
		(\pi^{*,N},c^{*,N})
		\right).
		\]
	\end{definition}
	We first consider the limiting case as $N\to\infty$. Before introducing the representative agent model, we impose the following assumption throughout this paper.
	\begin{assumption}\label{Ao}
		For each \(i=1,\ldots,N\), define the type vector by
		\[
		\xi^i:=
		\left(x_0^i,\lambda_i,b_i,\sigma_i,\sigma_i^0,\varepsilon_i,\gamma_i,\theta_i\right)
		\in \mathcal O
		:=
		\mathcal K\times[\underline\gamma,\overline\gamma]\times[0,1],
		\]
		where \(\mathcal K\) is a compact subset of \((0,\infty)^6\), and
		\(0<\underline\gamma<\overline\gamma<1\). Let \(\mathcal B(\mathcal O)\)
		denote the Borel \(\sigma\)-algebra on \(\mathcal O\). Assume that there exists a constant vector
		\[
		\xi:=
		\left(x_0,\lambda,b,\sigma,\sigma^0,\varepsilon,\gamma,\theta\right)
		\in \mathcal O
		\]
		such that
		\[
		\nu_0^N:=\frac1N\sum_{i=1}^N\delta_{\xi^i}
		\Rightarrow
		\nu_0:=\delta_\xi
		\quad\text{as }N\to\infty,
		\]
		where \(\nu_0^N\) denotes the empirical probability measure on
		\(\mathcal B(\mathcal O)\), and ``\(\Rightarrow\)'' denotes weak convergence,
		that is, for every bounded continuous function \(f\),
		\[
		\int_{\mathcal O} f\,d\nu_0^N
		\to
		\int_{\mathcal O} f\,d\nu_0
		\quad\text{as }N\to\infty.
		\]
	\end{assumption}
	As $N\to\infty$, Assumption \ref{Ao} leads to the following SDE of the wealth process for the representative agent, 
	\begin{equation}\label{SDElim}
		\left\{
		\begin{aligned}
			&dX_t=\left(\pi_t b-c_t\right)X_t dt+\pi_t X_t\sigma dW_t+\pi_t X_t\sigma^0dW_t^0-\pi_t X_{t-}dM_t \\
			&X_0=x_0
		\end{aligned}
		\right.
	\end{equation}
	where $W=(W_t)_{t\in[0,T]}$ is a scalar Brownian motion independent of the Brownian motions $(W^0,W^1,\ldots,W^N)$, and $M=(M_t)_{t\in[0,T]}$ is given by $M_t=N_t-\lambda t$, with $N=(N_t)_{t\in[0,T]}$ being a Poisson process with intensity $\lambda$. Let $\mathbb{G}^M=(\mathcal{G}^M_t)_{t\in[0,T]}$ be the filtration generated by $W$, $W^0$, and $N$. Similar to Definition \ref{mathcalAi}, the admissible control set for the representative agent $\mathcal{A}_M$ is defined analogously, replacing the filtration $\mathbb{G}$ with $\mathbb{G}^M$.
	
	Since the influence of any single agent's strategy pair on the population is negligible, we replace the geometric averages of the wealth and consumption with given processes $(m,\Gamma)$, and the objective functional of the representative agent associated with an admissible control $(\pi,c)\in\mathcal A_M$ is defined by
	\begin{equation}\label{objectlim}
		J(\pi,c)=\mathbb{E}\left[\int_0^T U(c_t X_t(\Gamma_tm_t)^{-\theta};\gamma)dt+\varepsilon U(X_T m_T^{-\theta};\gamma)\right].
	\end{equation}
	
	In the mean field limit, agents are expected to exhibit symmetric and identical behavior. Consequently, the geometric mean of wealth and consumption $(m,\Gamma)$ must coincide with the conditional geometric averages generated by the optimal response of the representative agent. This idea leads to the consistency condition in mean field games, which forms the basis for the definition of a mean field equilibrium.
	\begin{definition}[Mean field equilibrium (MFE)]\label{MFE}
		For given $\mathbb{F}^0$-adapted processes $m$ and $\Gamma$, let $\pi^{*,m,\Gamma}=(\pi_t^{*,m,\Gamma})_{t\in[0,T]}$ and $c^{*,m,\Gamma}=(c_t^{*,m,\Gamma})_{t\in[0,T]}$ be the optimal control for the control problem \eqref{SDElim}--\eqref{objectlim}. Then, the pair $\bigl(\pi^{*,m^*,\Gamma^*},c^{*,m^*,\Gamma^*}\bigr)$ is called a mean field equilibrium if there exist $(m^*,\Gamma^*)$ such that
		\begin{enumerate}
			\item[(i)] $\bigl(\pi^{*,m^*,\Gamma^*},c^{*,m^*,\Gamma^*}\bigr)$ is an optimal control pair for the control problem \eqref{SDElim}--\eqref{objectlim} associated with $(m^*,\Gamma^*)$;
			
			\item[(ii)] $(m^*,\Gamma^*)$ satisfies the consistency condition
			\begin{equation}\label{consistency}
				\left\{
				\begin{aligned}
					m_t^*&=\exp\left\{\mathbb{E}\left[\ln\left(X_t^{*,m^*,\Gamma^*}\right)\mid\mathcal{F}_t^0\right]\right\},
					\quad t\in[0,T],\\
					\Gamma_t^*&=\exp\left\{\mathbb{E}\left[\ln\left(c_t^{*,m^*,\Gamma^*}\right)\mid\mathcal{F}_t^0\right]\right\},\quad t\in[0,T],
				\end{aligned}
				\right.
			\end{equation}
			where $X^{*,m^*,\Gamma^*}=(X_t^{*,m^*,\Gamma^*})_{t\in[0,T]}$ is the wealth process associated with $\pi^{*,m^*,\Gamma^*}$ and $c^{*,m^*,\Gamma^*}$.
		\end{enumerate}
		If, in addition, both $\pi^{*,m^*,\Gamma^*}$ and $c^{*,m^*,\Gamma^*}$ are deterministic, then $\bigl(\pi^{*,m^*,\Gamma^*},c^{*,m^*,\Gamma^*}\bigr)$ is called a deterministic mean field equilibrium.
	\end{definition}
	
	The continuous-time consumption and investment problem has been a central topic in mathematical finance since the seminal contribution of Merton \cite{Merton1969,Merton1971,Merton1976}. In these papers, Merton established the foundational framework for consumption and portfolio management, formulated the classical expected utility maximization problem, and incorporated jump-diffusion dynamics into asset price modeling. This line of research has since been extensively developed in multiple directions; see, for example, \cite{Cox1989, Davis1990, Dammon2001}. In the classical Merton model, agents focus only on their own performance, while in real markets, investment decisions are often influenced by peers. The economic motivation for relative performance concerns dates back to Veblen's theory of conspicuous consumption \cite{Veblen1899} and Duesenberry's relative income hypothesis \cite{Duesenberry1949}, and was later formalized by Abel \cite{Abel1990} through utility specifications involving relative consumption. Such criteria have been widely used to model competition among investors, fund managers, and financial institutions, see \cite{Brown1996,Chevalier1997,Kempf2008}. Since relative performance concerns couple agents' decisions, the resulting problems are naturally formulated as stochastic differential games. Several related contributions have considered the finite-player setting \cite{Browne2000,Basak2014,Espinosa2015}.

	As the number of players increases, the coupling among agents through relative performance terms may cause classical methods to suffer from the curse of dimensionality, making the direct characterization of Nash equilibria generally intractable. To overcome this difficulty, Lasry and Lions \cite{Lasry2007} introduced mean field game (MFG) theory, while a closely related approach was independently developed by Huang et al. \cite{HuangCainesMalhame2006}. The key idea of MFG theory is that, when the number of players is sufficiently large, the coupling interactions among individuals can be replaced by the interaction between a representative agent and the population distribution, which substantially reduces the complexity of the original game and provides a powerful framework for analyzing large-population systems. 
	So far, the MFG framework has since been applied to a wide range of areas; we refer the reader to \cite{Gueant2011,Carmona2018a,Carmona2018b,Lauriere2024} and the references therein. In large-population financial models with relative performance concerns, the MFG approach has proved particularly useful. Lacker and Zariphopoulou \cite{Lacker2019} studied optimal investment games under relative performance criteria and constructed explicit equilibrium strategies for both finite-agent games and their mean field limits under CARA and CRRA utilities. Lacker and Soret \cite{Lacker2020} extended this framework to a consumption-investment model and derived closed-form solutions for both the finite-agent game and the corresponding MFG under CRRA utilities. Bo et al. \cite{Bo2024b} studied optimal investment and consumption problems for large populations under external habit formation. They derived mean field equilibria under both linear and multiplicative external habit formation and showed that approximate Nash equilibria for the finite-agent games can be constructed from the corresponding mean field equilibria. Liang and Zhang \cite{Liang2024} further investigated an optimal investment-consumption problem with heterogeneous agents, where the average habit formation and average wealth of peers serve as performance benchmarks. They obtained closed-form mean field equilibria and constructed approximate Nash equilibria for the corresponding finite-agent game. 
	
	The above studies are mainly developed within the classical continuous-time diffusion framework without incorporating jump risk in asset prices. This limitation is relevant in financial markets, where asset prices may exhibit abrupt downward movements due to market crashes, liquidity shocks, or unexpected macroeconomic events. In the spirit of Merton's jump-diffusion framework, incorporating downward jump risk is thus a natural extension, especially for optimal investment and consumption problems with relative performance concerns. The inclusion of jumps also brings new challenges to the mean field game analysis. Benazzoli et al. \cite{Benazzoli2020} established the existence of Nash equilibria for MFGs with controlled jump-diffusion processes by using relaxed controls and martingale problem methods. Benazzoli et al. \cite{Benazzoli2019} further constructed approximate Nash equilibria for large finite-player games and derived convergence rates. In a financial mean field game with jump risk, Bo et al. \cite{Bo2024a} studied an optimal portfolio problem with common noise and contagious jump risk modeled by a nonlinear Hawkes process under CRRA relative performance, derived a deterministic MFE, and constructed an approximate Nash equilibrium for the finite-agent game.
	
	Another feature relevant to the present model is the interaction through agents' controls. In classical MFGs, population interactions are typically described through the distribution of agents' states, whereas mean field games of controls allow such interactions to depend additionally on their controls. This class of models is also referred to as extended mean field games in part of the literature. Gomes and Voskanyan \cite{Gomes2016} developed an extended deterministic mean field game framework incorporating the collective behavior of the population. Cardaliaguet and Lehalle \cite{Cardaliaguet2018} studied a mean field game of controls motivated by trade crowding. The convergence of large finite-player games to MFGs with interactions through controls was investigated by Lauri\`ere and Tangpi \cite{Lauriere2022}, while Djete \cite{Djete2023} established general existence and convergence results for mean field games of controls, including models with common noise. 
	
	Motivated by Lacker and Soret \cite{Lacker2020}, we incorporate downward jump risk into a relative consumption--investment game with both idiosyncratic and common noise. A central contribution of this paper is the development of a more realistic MFG model for joint consumption and investment under jump risk. Compared with the diffusion model of Lacker and Soret \cite{Lacker2020}, the presence of jump risk introduces additional nonlinearities into the equilibrium investment condition and makes the explicit characterization of a Nash equilibrium for the finite-player game generally intractable. Compared with the single-control investment model with jump risk studied by Bo et al. \cite{Bo2024a}, our model incorporates consumption as an additional control, while the population consumption strategies enter directly into the running relative-performance criterion. The resulting problem is therefore a two-control aggregative mean field game of controls with common noise and downward jump risk. To address the resulting difficulties, we proceed in two steps. First, under Assumption \ref{Ao}, we formulate the limiting optimal control problem for a representative agent and derive an explicit deterministic mean field equilibrium by applying the stochastic maximum principle. In particular, when the jump intensity vanishes, our result recovers the constant-type mean field equilibrium of Lacker and Soret \cite{Lacker2020}. We then examine the quantitative properties and parameter sensitivity of the resulting equilibrium. Finally, based on the obtained mean field equilibrium, we introduce an auxiliary control problem for each agent in the heterogeneous finite-player game and prove that the resulting strategy profile constitutes a \(u_N\)-Nash equilibrium, with \(u_N\to0\) as \(N\to\infty\).
	
	The main results of this paper are stated as follows.
	
	\begin{thm}\label{th3}
		There exists a deterministic MFE strategy $(\pi^*,c^*)\in\mathcal{A}_{M}$ such that
		$\pi_t^*=\pi^*$ for all $t\in[0,T]$, where the constant $\pi^*$ satisfies
		\[
		\begin{aligned}
			(\gamma-1)\left(\sigma^2+\left(\sigma^0\right)^2\right)\pi^*
			-\theta\gamma\left(\sigma^0\right)^2 \pi^* 
			-\lambda\left(\left(1-\pi^*\right)^{\gamma-1}-1\right)+b=0,
		\end{aligned}
		\]
		and
		\[
		c_t^*
		=
		\frac{
			\rho\varepsilon^{\frac{1}{\gamma(1-\theta)-1}}
			e^{\frac{\rho}{\gamma(1-\theta)-1}(t-T)}
		}{
			\rho
			+
			(\gamma(1-\theta)-1)
			\varepsilon^{\frac{1}{\gamma(1-\theta)-1}}
			\left(1-e^{\frac{\rho(t-T)}{\gamma(1-\theta)-1}}\right)
		},
		\qquad t\in[0,T],
		\]
		where
		\begin{equation}\label{rhoeq}
			\begin{aligned}
				\rho
				=&\left(1-\gamma(1-\theta)\right)(b+\lambda)\pi^*
				+\theta\gamma\lambda\ln(1-\pi^*)
				-\frac12\Bigl[\theta\gamma+(\gamma-1)(\gamma-2)\Bigr]\sigma^2(\pi^*)^2 \\
				&-\frac12\left(\gamma(1-\theta)-1\right)
				\left(\gamma(1-\theta)-2\right)(\sigma^0)^2(\pi^*)^2
				-\left[(1-\pi^*)^{\gamma-1}-1\right]\lambda .
			\end{aligned}
		\end{equation}
		The corresponding fixed point $\left( m^*,\Gamma^* \right)$ satisfying the consistency condition \eqref{consistency} for $t\in[0,T]$ is characterized by
		\begin{equation}\label{m^*}
			\left\{
			\begin{aligned}
				m^*_t
				&=x_0 \exp\left\{\int_0^t\left(\eta\left(\pi^*\right)-c_s^*-\frac{1}{2}\left(\sigma^0\pi^*\right)^2\right)ds+\int_0^t\sigma^0\pi^*dW_s^0\right\},\\
				\Gamma_t^*
				&=c_t^*,
			\end{aligned}
			\right.
		\end{equation}
		where
		\[
		\eta(x)=(b+\lambda)x-\frac12\sigma^2x^2+\lambda\ln(1-x).
		\]
	\end{thm}
	
	\begin{thm}\label{Nashequ}
		Let Assumption \ref{Ao} hold. Then the strategy profile
		\((\bm\pi^*,\bm c^*)\) given by \eqref{pi*c*} is a
		\(u_N\)-Nash equilibrium for the \(N\)-agent game, with \(u_N\to0\) as \(N\to\infty\). More precisely, for every
		\(i=1,\ldots,N\),
		\begin{equation}\label{supJ}
			\sup_{(\pi^i,c^i)\in\mathcal A_i}J_i\left((\pi^i,c^i),(\bm\pi^*,\bm c^*)^{-i}\right)\leqslant J_i\left((\pi^{*,i},c^{*,i}),(\bm\pi^*,\bm c^*)^{-i}\right)+u_N.
		\end{equation}
	\end{thm}
	
	This paper is organized as follows.
	Section 2 provides the proof of Theorem \ref{th3} and discusses the case where the jump intensity vanishes. Section 3 presents numerical experiments to illustrate several quantitative properties and sensitivity results of the MFE. In Section 4, we construct an approximate Nash equilibrium for the $N$-player game based on the MFE of the representative agent.
	\section{Mean field equilibrium for the representative agent}
	Following Definition \ref{MFE}, we first solve the representative agent's optimal control problem with the $\mathbb{F}^0$-adapted processes $m$ and $\Gamma$ fixed. We then search for a fixed point satisfying the consistency conditions, thereby deriving the corresponding MFE. The detailed proof of Theorem \ref{th3} is presented below.
	\begin{proof}[Proof of Theorem~\ref{th3}]
		For fixed processes $m_t$ and $\Gamma_t$, by the Stochastic Maximum Principle (SMP) in \cite{Oksendal2014, Oksendal2019}, for $(t,x,\pi,c,p,q,q^0,y)\in[0,T]\times\mathbb{R}_+\times(-\infty,1-\epsilon_0]\times\mathbb{R}_+\times\mathbb{R}^4$, we introduce the Hamiltonian
		\begin{equation}\label{Hamiltonian}
			\begin{aligned}
				H(t,x,\pi,c,p,q,q^0,y):=&U\left(cx\left(\Gamma m\right)^{-\theta};\gamma\right)+\left( \pi b-c \right)xp\\&+\pi x\sigma q+\pi x\sigma^0 q^0-\pi x \lambda y.
			\end{aligned}
		\end{equation}
		Let $\pi^{m,\Gamma}$ and $c^{m,\Gamma}$ be admissible strategies depending on $m$ and $\Gamma$, and let $X^{m,\Gamma}$ denote the corresponding wealth process under $(\pi^{m,\Gamma},c^{m,\Gamma})$. Then, we reformulate the optimal control problem \eqref{SDElim}--\eqref{objectlim} as the following associated dual FBSDE:
		\begin{equation}\label{FBSDE_P}
			\left\{
			\begin{aligned}
				dX_t^{m,\Gamma}&=\left(\pi^{m,\Gamma}_t b-c^{m,\Gamma}_t\right)X^{m,\Gamma}_t dt+\pi^{m,\Gamma}_t X^{m,\Gamma}_t\sigma dW_t\\
				&\quad+\pi^{m,\Gamma}_t X^{m,\Gamma}_t\sigma^0dW_t^0-\pi^{m,\Gamma}_t X^{m,\Gamma}_{t-}dM_t,\\
				dP_t^{m,\Gamma}&= -\Bigg[ \frac{c^{m,\Gamma}_t}{\left(\Gamma m\right)^\theta}\left( \frac{c^{m,\Gamma}_t X^{m,\Gamma}_t}{\left(\Gamma m\right)^\theta} \right)^{\gamma-1}-c^{m,\Gamma}_t P^{m,\Gamma}_t\\ 
				&\quad+\pi^{m,\Gamma}_t \left( bP^{m,\Gamma}_t+\sigma Q^{m,\Gamma}_t+\sigma^0 Q_t^{0,m,\Gamma}-\lambda Y^{m,\Gamma}_t \right) \Bigg]dt\\
				&\quad+Q_t^{m,\Gamma}dW_t+Q_t^{0,m,\Gamma}dW^0_t+Y^{m,\Gamma}_t dM_t,\\
				X_0^{m,\Gamma}&=x_0,\\
				P_T^{m,\Gamma}&=\varepsilon(X_T^{m,\Gamma})^{\gamma-1}m_T^{-\theta\gamma}.
			\end{aligned}
			\right.
		\end{equation}
		Let $m_t,\Gamma_t$ be of the following form for $t\in[0,T]$:
		\begin{equation}\label{mGamma}
			m_t=\exp\left\{ \mathbb{E}\left[ \ln X_t^{m,\Gamma}|\mathcal{F}^0_t \right] \right\},\quad\Gamma_t=\exp\left\{\mathbb{E}\left[\ln c^{m,\Gamma}_t|\mathcal{F}^0_t\right]\right\}.
		\end{equation}
		Then, by It\^{o}'s formula, we have
		\begin{equation}\label{SDEm}
			\begin{aligned}
				dm_t&=m_t\left\{\left(b+\lambda\right)\mathbb{E}\left[\pi^{m,\Gamma}_t|\mathcal{F}_t^0\right]-\frac{1}{2}\left(\sigma^2+\left(\sigma^0\right)^2\right)\mathbb{E}\left[(\pi^{m,\Gamma}_t)^2|\mathcal{F}_t^0\right]-\mathbb{E}\left[c^{m,\Gamma}_t|\mathcal{F}_t^0\right]\right.\\ &\left.\quad+\frac{1}{2}\left( \sigma^0\mathbb{E}\left[ \pi^{m,\Gamma}_t|\mathcal{F}_t^0 \right] \right)^2+\lambda\mathbb{E}\left[ \ln\left( 1-\pi^{m,\Gamma}_{t-} \right)|\mathcal{F}_t^0 \right]\right\}dt+\sigma^0m_t\mathbb{E}\left[\pi^{m,\Gamma}_t|\mathcal{F}_t^0\right]dW_t^0\\
				&=:m_t\hat{\eta}(t,\pi^{m,\Gamma},c^{m,\Gamma})dt+\sigma^0m_t\mathbb{E}\left[\pi^{m,\Gamma}_t|\mathcal{F}_t^0\right]dW_t^0.
			\end{aligned}
		\end{equation}
		From the terminal condition of $P_t^{m,\Gamma}$, we assume that
		\begin{equation}\label{AssP}
			P^{m,\Gamma}_t=\varepsilon (X^{m,\Gamma}_t)^{\gamma-1} m_t^{-\theta \gamma}\varphi_t,\quad\varphi_T=1.
		\end{equation}
		With this assumption, the SDE of $X^{m,\Gamma}_t$ and \eqref{SDEm}, by It\^{o}'s formula, we have
		\begin{equation}\label{SDEP}
			\begin{aligned}
				dP^{m,\Gamma}_t=&P^{m,\Gamma}_t\Bigg\{\frac{\varphi_t'}{\varphi_t}+\left( \gamma-1 \right)\left(\left(b+\lambda\right)\pi^{m,\Gamma}_t-c^{m,\Gamma}_t\right)-\theta\gamma\hat{\eta}(t,\pi^{m,\Gamma},c^{m,\Gamma})\\
				&-\theta\gamma\left( \gamma-1 \right)\left(\sigma^0\right)^2\pi^{m,\Gamma}_t\mathbb{E}\left[\pi^{m,\Gamma}_t|\mathcal{F}_t^0\right]+\frac{1}{2}(\gamma-1)(\gamma-2)\left(\sigma^2+\left(\sigma^0\right)^2\right)\left(\pi^{m,\Gamma}_t\right)^2\\&+\frac{1}{2}\theta\gamma(\theta\gamma+1)\left(\sigma^0\right)^2\left( \mathbb{E}\left[\pi^{m,\Gamma}_t|\mathcal{F}_t^0\right] \right)^2+\left[\left(1-\pi^{m,\Gamma}_t\right)^{\gamma-1}-1\right]\lambda \Bigg\} dt\\&+(\gamma-1)P^{m,\Gamma}_t\pi^{m,\Gamma}_t\sigma dW_t+P^{m,\Gamma}_t\sigma^0\left[ (\gamma-1)\pi^{m,\Gamma}_t-\theta\gamma\mathbb{E}\left[\pi^{m,\Gamma}_t|\mathcal{F}_t^0\right] \right]dW_t^0\\
				&+P^{m,\Gamma}_t\left\{ \left(1-\pi^{m,\Gamma}_t\right)^{\gamma-1}-1 \right\}dM_t.
			\end{aligned}
		\end{equation}
		Let $\left(\pi^{*,m,\Gamma},c^{*,m,\Gamma}\right)$ be the candidate optimal control with given $(m,\Gamma)$, and let $\left(X^{*,m,\Gamma}_t,P^{*,m,\Gamma},Q^{*,m,\Gamma},Q^{0,*,m,\Gamma},Y^{*,m,\Gamma}\right)$ be the corresponding solution to the FBSDE \eqref{FBSDE_P}. Since the Hamiltonian is linear in \(\pi\) and strictly concave in \(c\in(0,\infty)\), the maximum condition leads to
		\begin{equation}\label{HCondition}
			\left\{
			\begin{aligned}
				b P^{*,m,\Gamma}_t+\sigma Q^{*,m,\Gamma}_t+\sigma^0 Q^{0,*,m,\Gamma}_t-\lambda Y^{*,m,\Gamma}_t&=0,\\
				\left( \frac{X^{*,m,\Gamma}_t c^{*,m,\Gamma}_t}{\left( \Gamma m \right)^{\theta}} \right)^{\gamma-1}\frac{X^{*,m,\Gamma}_t}{\left( \Gamma m \right)^{\theta}}-X^{*,m,\Gamma}_t P^{*,m,\Gamma}_t&=0.
			\end{aligned}
			\right.
		\end{equation}
		Comparing two SDEs of $P_t^{*,m,\Gamma}$ in \eqref{FBSDE_P} and \eqref{SDEP}, we obtain that
		\begin{equation}\label{QY}
			\left\{
			\begin{aligned}
				Q_t^{*,m,\Gamma}&=(\gamma-1)\sigma P_t^{*,m,\Gamma}\pi_t^{*,m,\Gamma},\\
				Y_t^{*,m,\Gamma}&=P_t^{*,m,\Gamma}\left[ \left( 1-\pi_t^{*,m,\Gamma} \right)^{\gamma-1}-1 \right],\\
				Q_t^{0,*,m,\Gamma}&=P_t^{*,m,\Gamma}\left\{ (\gamma-1)\pi_t^{*,m,\Gamma}\sigma^0-\theta\gamma\sigma^0 \mathbb{E}\left[ \pi^{*,m,\Gamma}_t|\mathcal{F}_t^0 \right] \right\}.
			\end{aligned}
			\right.
		\end{equation}
		By the first equation in \eqref{HCondition} and \eqref{QY}, it yields that
		$$
		\begin{aligned}
			(\gamma-1)\left(\sigma^2+\left(\sigma^0\right)^2\right)\pi_t^{*,m,\Gamma}-\theta\gamma\left(\sigma^0\right)^2 \mathbb{E}\left[ \pi^{*,m,\Gamma}_t|\mathcal{F}_t^0 \right]&\\
			-\lambda\left(\left(1-\pi_t^{*,m,\Gamma}\right)^{\gamma-1}-1\right)+b&=0.
		\end{aligned}
		$$
		Since we focus on a deterministic MFE, the above expression reduces to
		\begin{equation}\label{equpi}
			\begin{aligned}
				(\gamma-1)\left(\sigma^2+\left(\sigma^0\right)^2\right)\pi_t^{*,m,\Gamma}-\theta\gamma\left(\sigma^0\right)^2 \pi_t^{*,m,\Gamma}&\\
				-\lambda\left(\left(1-\pi_t^{*,m,\Gamma}\right)^{\gamma-1}-1\right)+b&=0
			\end{aligned}
		\end{equation}
		By Bo et al. \cite[Lemma 2.2]{Bo2024a}, the above equation admits a unique solution
		\(\pi_t^{*,m,\Gamma}\in(0,1-\epsilon_0]\) for all \(t\in[0,T]\), where \(\epsilon_0\in(0,1)\) is sufficiently small. Furthermore, this solution is independent of $m$ and $\Gamma$, and depends only on the constant parameter vector $\xi$. Hence,
		\[
		\pi_t^{*,m,\Gamma}=\pi^*=\phi(\xi),\quad t\in[0,T],
		\]
		where $\phi$ is a Lipschitz continuous function. Moreover, it is easy to verify that $\pi^*$ satisfies the
		admissibility condition. Thus, for $t\in[0,T]$, $\pi^*=\phi(\xi)$ is the best response control.
		
		Then, we proceed to the explicit form of the optimal control $c^{*,m,\Gamma}$. By comparing the $dt$-terms in \eqref{FBSDE_P} and \eqref{SDEP}, and using \eqref{HCondition}, we obtain the following ODE
		\begin{equation}\label{ODEphi}
			\frac{\varphi_t'}{\varphi_t}=\left( \gamma-1 \right)c_t^{*,m,\Gamma}-\theta\gamma \mathbb{E}[c_t^{*,m,\Gamma}|\mathcal{F}_t^0]+\rho,
		\end{equation}
		where
		\begin{equation}\label{rho}
			\begin{aligned}
				\rho
				=&\left(1-\gamma(1-\theta)\right)(b+\lambda)\pi^*
				+\theta\gamma\lambda\ln(1-\pi^*)-\frac12\Bigl[\theta\gamma+(\gamma-1)(\gamma-2)\Bigr]\sigma^2(\pi^*)^2 \\
				&-\frac12\left(\gamma(1-\theta)-1\right)\left(\gamma(1-\theta)-2\right)(\sigma^0)^2(\pi^*)^2-\left[(1-\pi^*)^{\gamma-1}-1\right]\lambda,
			\end{aligned}
		\end{equation}
		which depends only on the parameters of the control problem. We claim that \(\rho<0\). Indeed, it follows from \eqref{equpi} that,
		\[
		b=(1-\gamma)(\sigma^2+(\sigma^0)^2)\pi^*+\theta\gamma(\sigma^0)^2\pi^*
		+\lambda\big((1-\pi^*)^{\gamma-1}-1\big).
		\]
		Substituting this into \eqref{rho}, we obtain
		\begin{equation}\label{eq:rho1}
			\rho=-\frac{\gamma(\pi^*)^2}{2}\left([1-\gamma-\theta(1-2\gamma)]\sigma^2+(1-\theta)(1-\gamma+\gamma\theta)(\sigma^0)^2\right)+\lambda F(1-\pi^*),
		\end{equation}
		where
		\begin{equation}\label{eq:F}
			F(x)=1-a x^{\gamma-1}-(1-a)x^\gamma+\theta\gamma\ln x,\ x\in[\epsilon_0,1).
		\end{equation}
		and $a=\gamma(1-\theta)\in[0,1).$
		
		Since \(\gamma\in(0,1)\) and \(\theta\in[0,1]\),
		\[
		1-\gamma-\theta(1-2\gamma)\ge \min\{1-\gamma,\gamma\}>0.
		\]
		Hence, we have
		\begin{equation}\label{eq:1}
			-\frac{\gamma(\pi^*)^2}{2}\left([1-\gamma-\theta(1-2\gamma)]\sigma^2+(1-\theta)(1-\gamma+\gamma\theta)(\sigma^0)^2\right)<0.
		\end{equation}
		Since for each \(x>0\) and $\beta\in\mathbb{R}$,
		\[
		x^\beta\ge 1+\beta\ln x,
		\]
		it follows, by taking \(\beta=\gamma-1\) and \(\beta=\gamma\), that
		\[
		a x^{\gamma-1}+(1-a)x^\gamma
		\ge a\bigl(1+(\gamma-1)\ln x\bigr)+(1-a)\bigl(1+\gamma\ln x\bigr)
		=1+\theta\gamma\ln x.
		\]
		Consequently, \begin{equation}\label{eq:lamdaF}
			\lambda F(1-\pi^*)\le 0.
		\end{equation}
		Combining \eqref{eq:rho1}, \eqref{eq:1}, and \eqref{eq:lamdaF}, we conclude that
		\[
		\rho<0.
		\]
		
		From the expression of $P_t^{m,\Gamma}$ in \eqref{AssP} and the second equation in \eqref{HCondition}, we obtain
		\begin{equation}\label{phi}
			\begin{aligned}
				\varphi_t=\frac{\left(c_t^{*,m,\Gamma}\right)^{\gamma-1}}{\varepsilon\Gamma_t^{\theta\gamma}}.
			\end{aligned}
		\end{equation}
		Replacing \eqref{phi} into \eqref{ODEphi}, we can get the ODE for $c_t^{*,m,\Gamma}$ that
		\begin{equation}\label{ODEc}
			\left\{
			\begin{aligned}
				\frac{\bigl(c_t^{*,m,\Gamma}\bigr)'}{c_t^{*,m,\Gamma}}
				&=c_t^{*,m,\Gamma}+\frac{\theta\gamma}{\gamma-1}\left(\frac{\Gamma_t'}{\Gamma_t}-\mathbb{E}\left[c_t^{*,m,\Gamma}\mid\mathcal F_t^0\right]\right)+\frac{\rho}{\gamma-1},\\
				c_T^{*,m,\Gamma}
				&=\left(\varepsilon\,\Gamma_T^{\theta\gamma}\right)^{\frac{1}{\gamma-1}}.
			\end{aligned}
			\right.
		\end{equation}
		Since we focus on a deterministic MFE, we have,
		\begin{equation}\label{Gammac}
			\Gamma_t=c_t^{*,m,\Gamma},\quad t\in[0,T].
		\end{equation}
		Thus, the ODE in \eqref{ODEc} reduces to the following Bernoulli differential equation:
		\begin{equation}\label{ODEcast}
			\left\{
			\begin{aligned}
				&\frac{\left( c_t^{*,m,\Gamma} \right)'}{c_t^{*,m,\Gamma}}=c_t^{*,m,\Gamma}+\frac{\rho}{\gamma-\theta\gamma-1}\\
				&c_T^{*,m,\Gamma}=\varepsilon^{\frac{1}{\gamma-\theta\gamma-1}}
			\end{aligned}
			\right.
		\end{equation}
		Solving \eqref{ODEcast} yields
		\begin{equation*}
			\begin{aligned}
				c_t^{*,m,\Gamma}=\frac{\rho\varepsilon^{\frac{1}{\gamma(1-\theta)-1}}e^{\frac{\rho}{\gamma(1-\theta)-1}(t-T)}}{\rho+(\gamma(1-\theta)-1)\varepsilon^{\frac{1}{\gamma(1-\theta)-1}}(1-e^{\frac{\rho(t-T)}{\gamma(1-\theta)-1}})},\quad t\in[0,T].
			\end{aligned}
		\end{equation*}
		Clearly, $c_t^{*,m,\Gamma}$ is independent of $(m,\Gamma)$. Therefore, we simply write $c_t^{*,m,\Gamma}$ as $c_t^*$. Since $\rho<0$, one readily verifies that $c_t^*>0$ for all $t\in[0,T]$. In addition, Assumption \ref{Ao} implies that $c_t^*$ is bounded on $[0,T]$. Hence, $c^*=(c_t^*)_{t\in[0,T]}$ satisfies the
		admissibility condition, and therefore $c_t^*$ is the best response for $t\in[0,T]$.	
		
		Based on the ansatz for $P_t^*$ in \eqref{AssP}, the solutions to \eqref{QY} and \eqref{ODEphi} can be obtained explicitly, yielding closed-form expressions for $P_t^*$, $Q_t^*$, $Q_t^{0,*}$, and $Y_t^*$. Consequently, the adjoint processes are well defined. Moreover, since \eqref{mGamma} holds for any strategy $(\pi^{m,\Gamma},c^{m,\Gamma})\in\mathcal{A}_M$, the consistency conditions
		\begin{equation*}
			m_t^*
			=
			\exp\left\{
			\mathbb{E}\left[
			\ln X_t^{*}\mid\mathcal{F}^0_t
			\right]
			\right\},
			\qquad
			\Gamma_t^*
			=
			\exp\left\{
			\mathbb{E}\left[
			\ln c^{*}_t\mid\mathcal{F}^0_t
			\right]
			\right\}.
		\end{equation*}
		are readily verified.
		
		It then follows from \eqref{SDEm} and \eqref{Gammac} that
		\[
		\left\{
		\begin{aligned}
			m_t^*
			&=
			x_0
			\exp\left\{
			\int_0^t
			\left(
			\eta(\pi^*)-c_s^*
			-\frac12(\sigma^0\pi^*)^2
			\right)ds
			+
			\int_0^t\sigma^0\pi^*\,dW_s^0
			\right\},\\
			\Gamma_t^*
			&=
			c_t^*,
		\end{aligned}
		\right.
		\]
		where
		\begin{equation*}\label{eta}
			\eta(x)
			=
			(b+\lambda)x
			-\frac12\sigma^2x^2
			+\lambda\ln(1-x).
		\end{equation*}
		This completes the proof.
	\end{proof}
	\begin{rmk}[The case without jump risk]
		When the jump intensity vanishes, i.e., $\lambda=0$, the equilibrium investment strategy reduces to
		\[
		\pi^*
		=
		\frac{b}{
			(1-\gamma)\bigl(\sigma^2+(\sigma^0)^2\bigr)
			+\theta\gamma(\sigma^0)^2},
		\]
		which coincides with the equilibrium investment strategy obtained by Lacker and Soret \cite{Lacker2020} in the constant-type setting after identifying the corresponding notations. Moreover, the constant $\rho$ in \eqref{rhoeq} simplifies to
		\[
		\begin{aligned}
			\rho_0
			=&\left(1-\gamma(1-\theta)\right)b\pi^*
			-\frac12\Bigl[\theta\gamma+(\gamma-1)(\gamma-2)\Bigr]
			\sigma^2(\pi^*)^2  \\
			&-\frac12\left(\gamma(1-\theta)-1\right)
			\left(\gamma(1-\theta)-2\right)(\sigma^0)^2(\pi^*)^2 .
		\end{aligned}
		\]
		Using the same argument as in the proof of Theorem \ref{th3}, one readily verifies that $\rho_0<0$. Consequently, the equilibrium consumption strategy is given by
		\[
		c_t^*
		=
		\frac{
			\rho_0\varepsilon^{\frac{1}{\gamma(1-\theta)-1}}
			e^{\frac{\rho_0}{\gamma(1-\theta)-1}(t-T)}
		}{
			\rho_0
			+
			(\gamma(1-\theta)-1)
			\varepsilon^{\frac{1}{\gamma(1-\theta)-1}}
			\left(1-e^{\frac{\rho_0(t-T)}{\gamma(1-\theta)-1}}\right)
		},
		\qquad t\in[0,T].
		\]
		Again, after identifying the notation, this expression agrees exactly with the equilibrium consumption strategy derived in \cite{Lacker2020} for the constant-type case. Therefore, when $\lambda=0$, Theorem \ref{th3} recovers the mean field equilibrium of Lacker and Soret \cite{Lacker2020} with a constant type vector.
	\end{rmk}

	\section{Numerical analysis}
	In this section, we provide a numerical analysis of the deterministic MFE strategy $(\pi^*, c^*) \in \mathcal{A}_M$ derived in Theorem \ref{th3}. The sensitivity of the equilibrium investment strategy \(\pi^*\) to parameter changes follows directly from its explicit formula; see Bo et al. \cite[Lemma 3.1]{Bo2024a}. We therefore concentrate on the behavior of the equilibrium consumption policy $c_t^*$.
	\begin{table}[htbp]
		\centering
		\caption{Benchmark parameter values}
		\label{tab1}
		\renewcommand{\arraystretch}{1.15}
		\begin{tabularx}{0.94\textwidth}{
				>{\centering\arraybackslash}m{0.18\textwidth}
				>{\centering\arraybackslash}X
				>{\centering\arraybackslash}m{0.14\textwidth}
			}
			\toprule
			Model parameter & Financial meaning & Value \\
			\midrule
			$x_0$         & Initial wealth level                                         & $3.00$ \\
			$b$           & Expected return of risky assets                              & $0.80$ \\
			$\lambda$     & Poisson jump intensity                             & $0.50$ \\
			$\sigma$      & Idiosyncratic volatility of the stock                        & $0.80$ \\
			$\sigma^0$    & Common volatility                                            & $0.60$ \\
			$\gamma$      & $1-\gamma$ denotes the coefficient of relative risk aversion & $0.40$ \\
			$\theta$      & Competition weight                                           & $0.70$ \\
			$\varepsilon$ & Weight assigned to terminal wealth                           & $1.60$ \\
			$T$           & Terminal time (investment horizon)                           & $3.00$ \\
			\bottomrule
		\end{tabularx}
	\end{table}
	
	We begin by examining the sensitivity of the MFE consumption strategy $c_t^*$ to individual parameters, taking the benchmark values in Table \ref{tab1} as the baseline specification. The effects of parameter variations on $c_t^*$ can be classified into two categories. 
	
	First, as illustrated in Figure \ref{fig:c_sens_group1}, the consumption strategy $c_t^*$ increases with $\lambda$, $\theta$, $\sigma$, and $\sigma^0$. This behavior admits a natural economic interpretation. A higher jump intensity $\lambda$ increases the exposure of future wealth to adverse jump shocks, making investment less attractive and inducing the agent to consume more in the present. A larger competition weight parameter $\theta$ places greater weight on relative consumption performance, thereby strengthening the incentive for current consumption. In addition, higher values of the idiosyncratic volatility $\sigma$ and the common volatility $\sigma_0$ increase uncertainty in future wealth accumulation, prompting the agent to shift consumption toward the present in response to the elevated level of market risk.
	\begin{figure}[htbp]
		\centering
		\begin{subfigure}{0.47\textwidth}
			\centering
			\includegraphics[width=\textwidth]{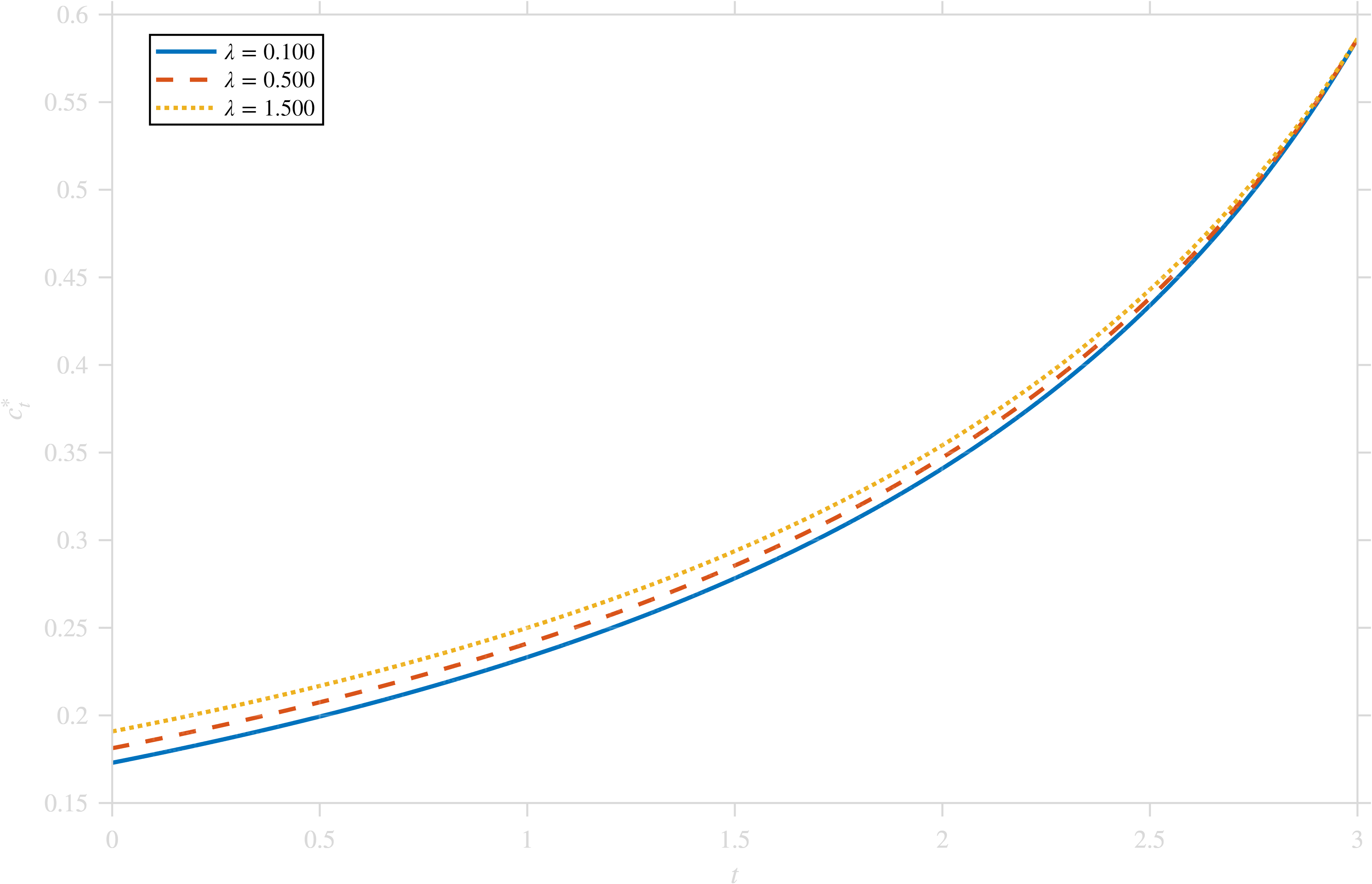}
		\end{subfigure}
		\hfill
		\begin{subfigure}{0.47\textwidth}
			\centering
			\includegraphics[width=\textwidth]{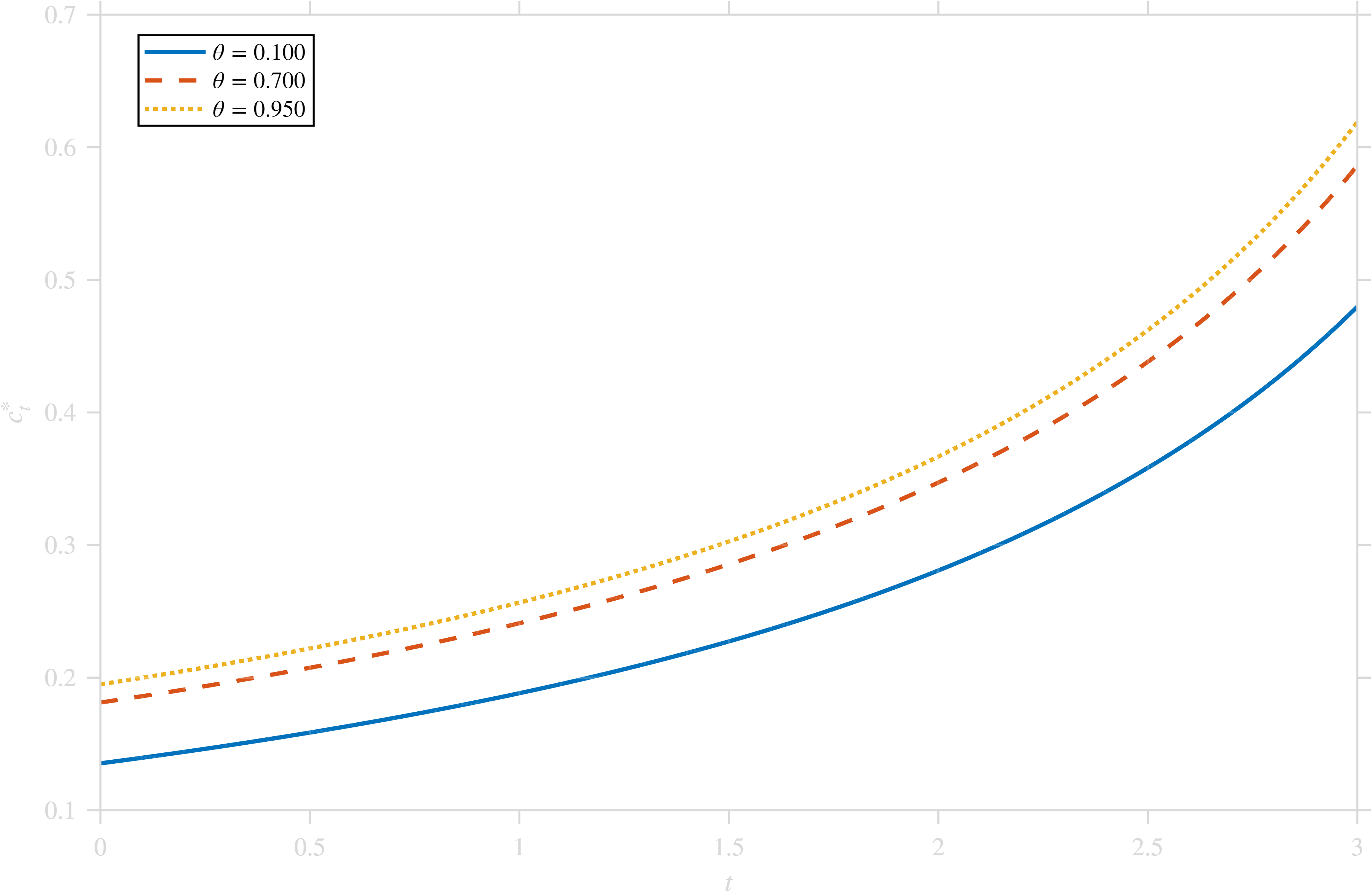}
		\end{subfigure}
		
		\vspace{0.35em}
		
		\begin{subfigure}{0.47\textwidth}
			\centering
			\includegraphics[width=\textwidth]{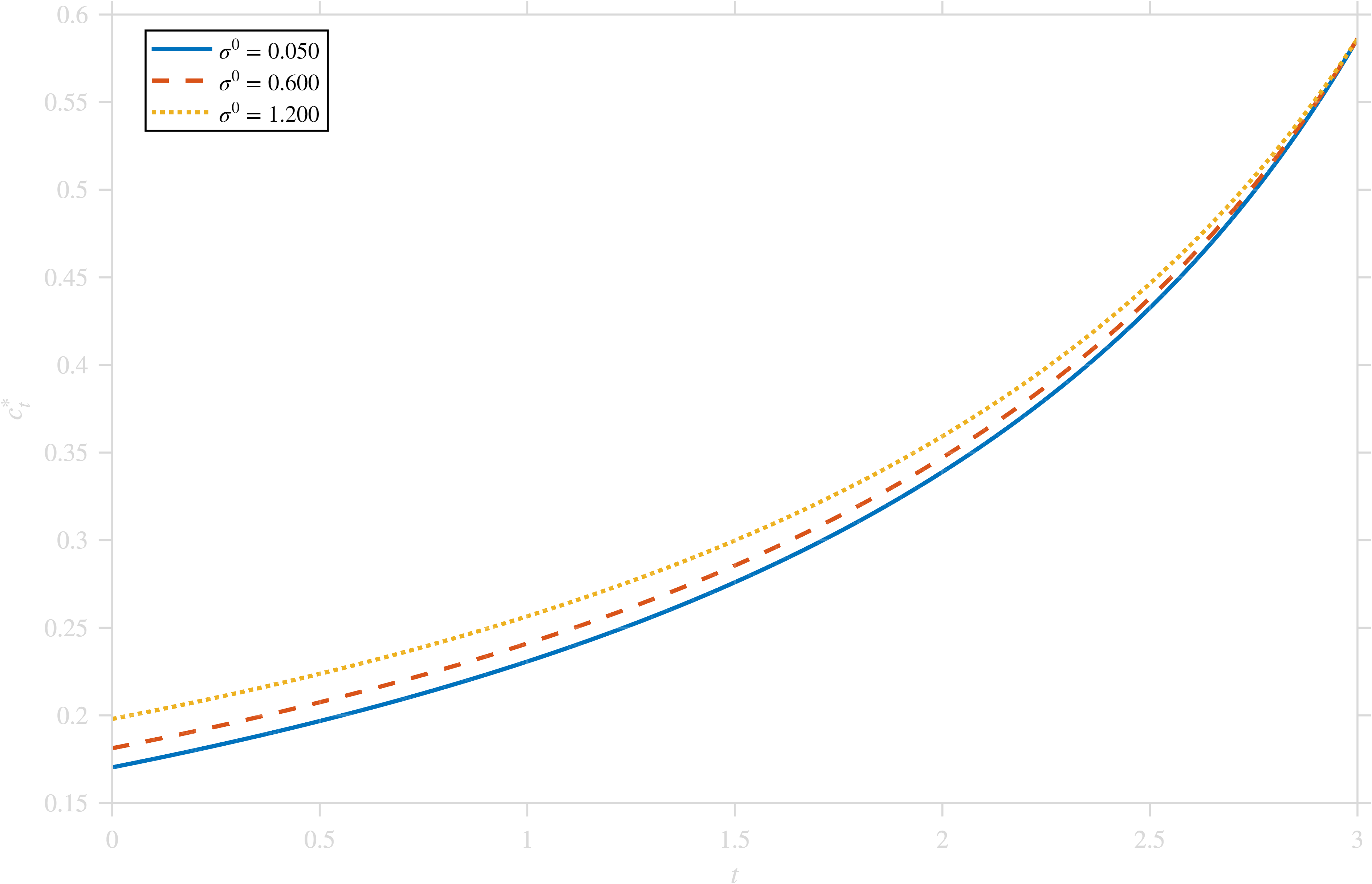}
		\end{subfigure}
		\hfill
		\begin{subfigure}{0.47\textwidth}
			\centering
			\includegraphics[width=\textwidth]{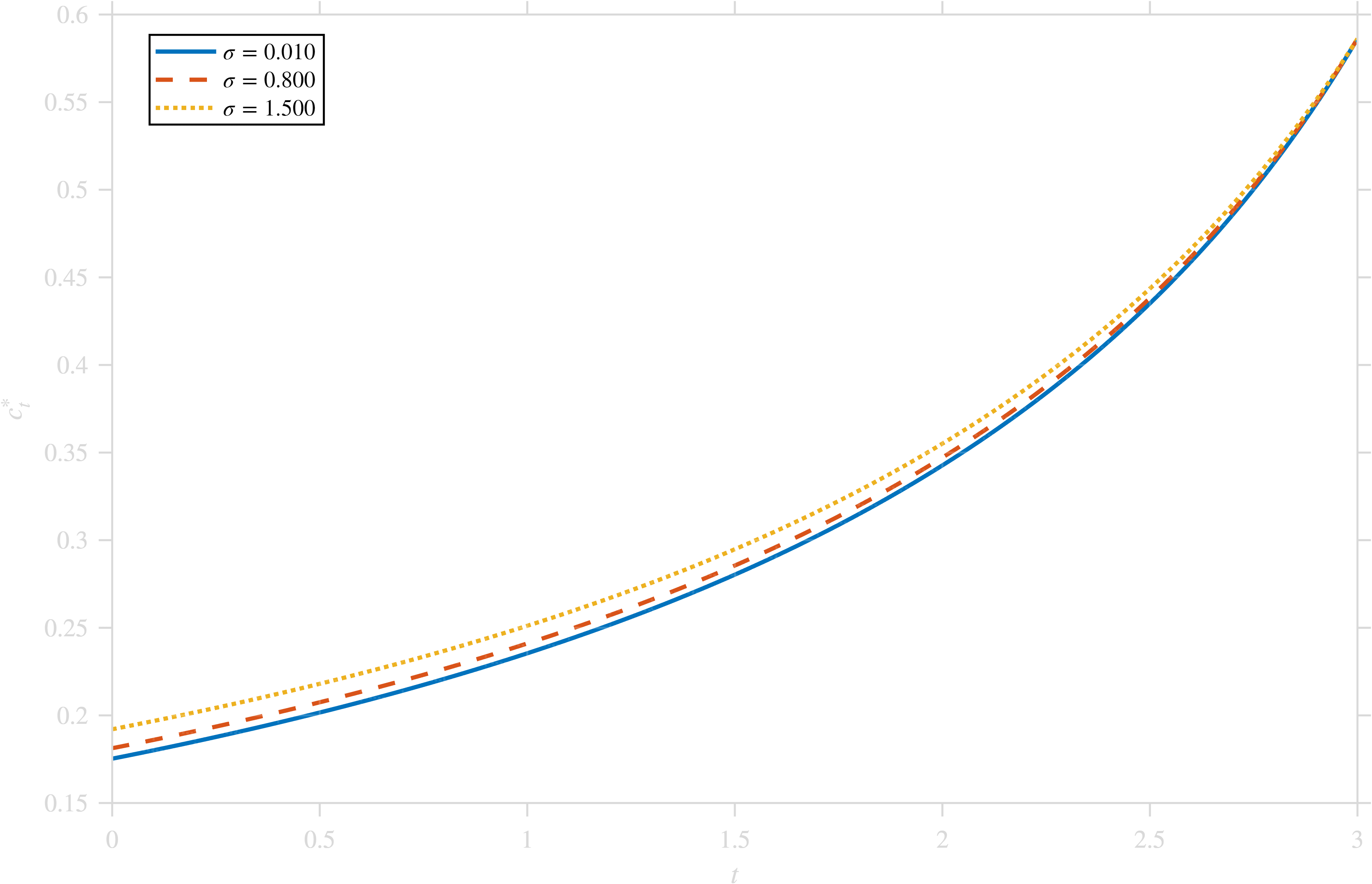}
		\end{subfigure}
		\caption{Mean field equilibrium strategy $c_t^*$ with respect to $\lambda$, $\sigma$, $\sigma^0$, and $\theta$}
		\label{fig:c_sens_group1}
	\end{figure}
	\begin{figure}[htbp]
		\centering
		\begin{subfigure}{0.32\textwidth}
			\centering
			\includegraphics[width=\textwidth]{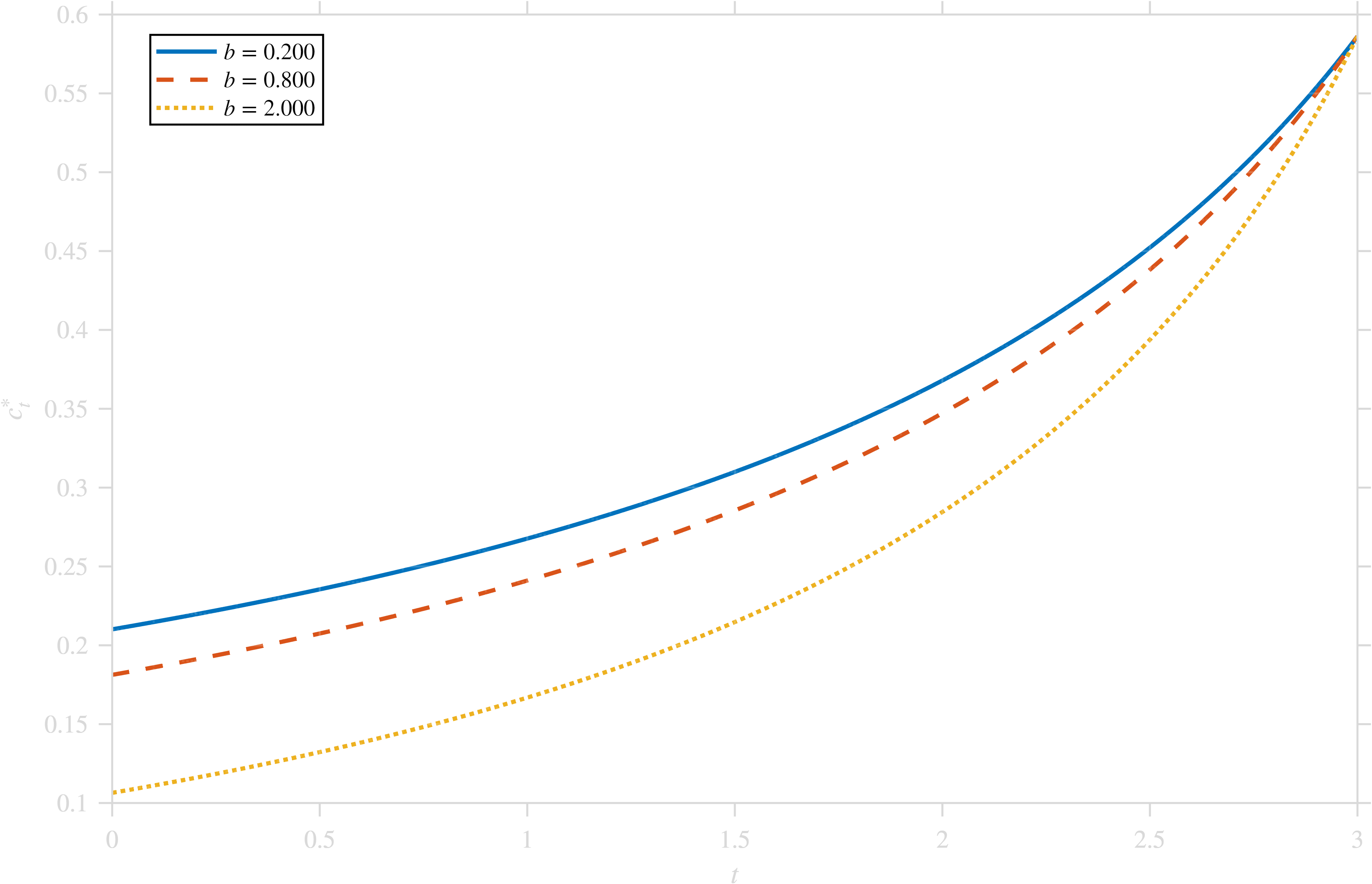}
		\end{subfigure}
		\hfill
		\begin{subfigure}{0.32\textwidth}
			\centering
			\includegraphics[width=\textwidth]{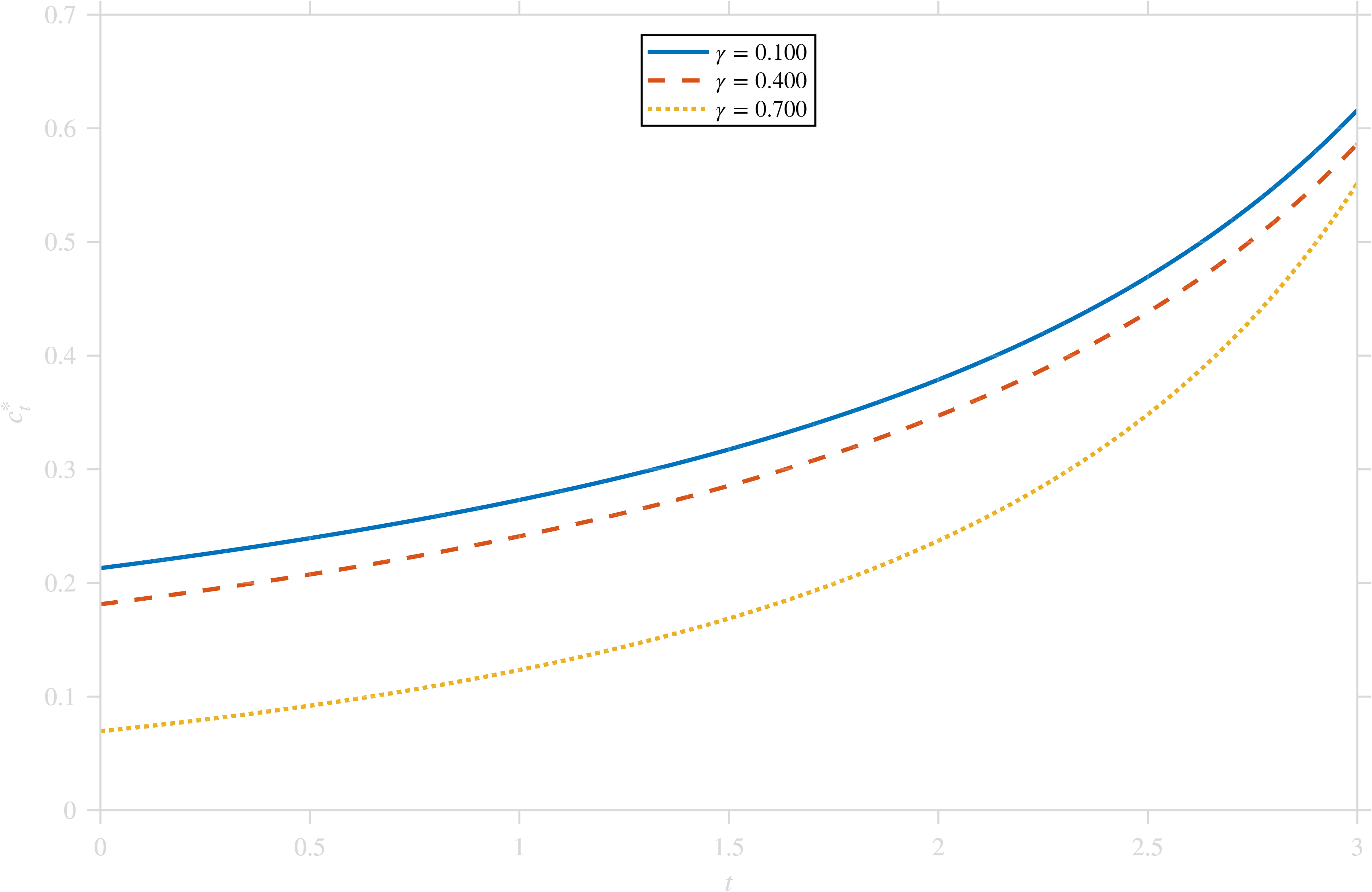}
		\end{subfigure}
		\hfill
		\begin{subfigure}{0.32\textwidth}
			\centering
			\includegraphics[width=\textwidth]{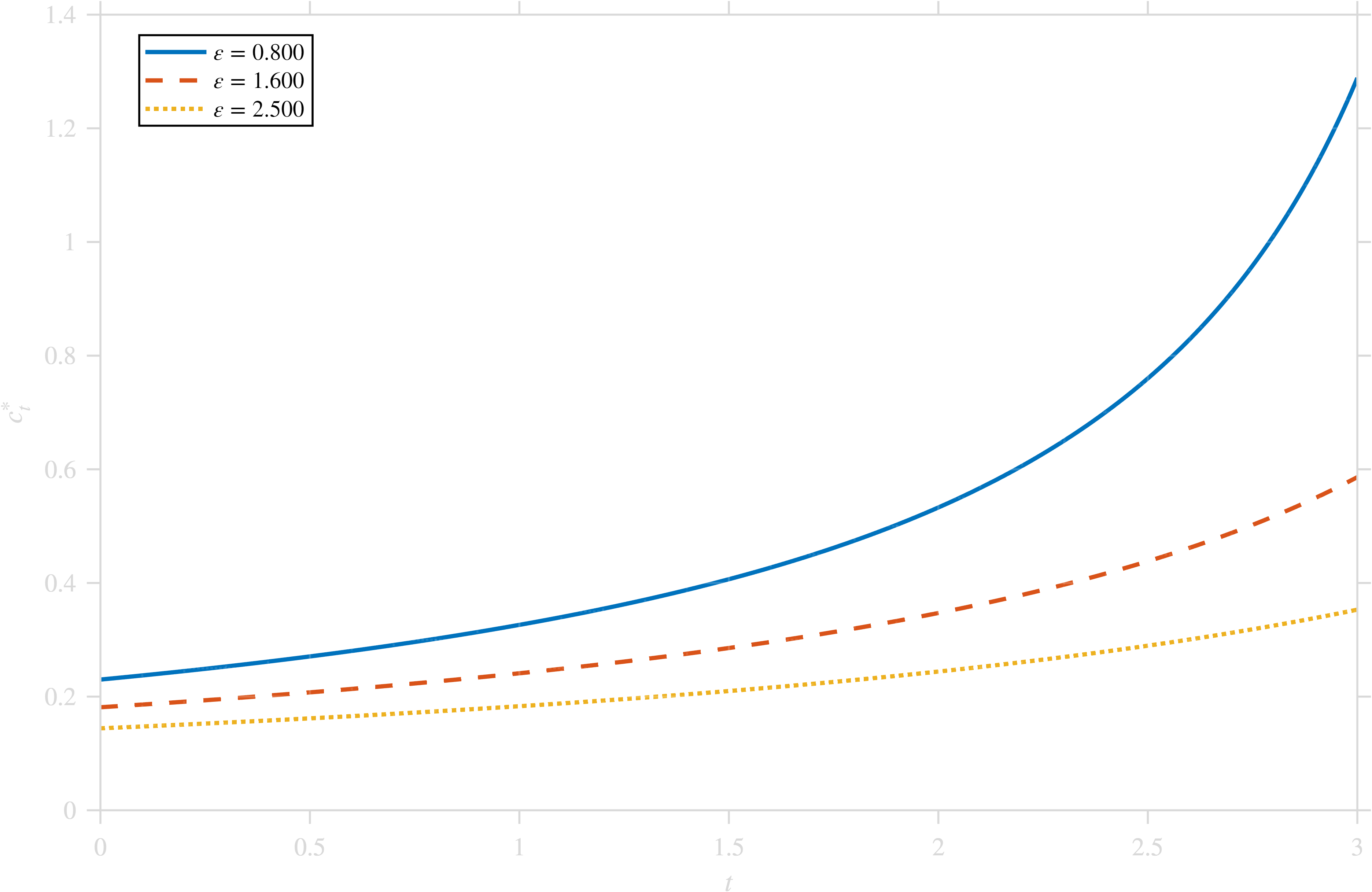}
		\end{subfigure}
		\caption{Mean field equilibrium strategy $c_t^*$ with respect to $b$, $\gamma$, and $\varepsilon$}
		\label{fig:c_sens_group2}
	\end{figure}
	
	By contrast, Figure \ref{fig:c_sens_group2} shows that $c_t^*$ decreases with $b$, $\gamma$, and $\varepsilon$. These monotone relationships are also consistent with standard economic intuition. A larger value of $b$ corresponds to a higher expected return on the risky asset, making investment opportunities more attractive and thereby reducing the incentive for current consumption. Furthermore, a smaller $\gamma$ corresponds to a higher degree of risk aversion. In this case, the agent adopts a more conservative investment strategy, which leads to a higher consumption rate. Finally, the effect of the terminal weight parameter $\varepsilon$ is straightforward: A larger value of $\varepsilon$ increases the importance of terminal wealth in the optimization objective and consequently lowers current consumption.
	
	In particular, all the curves of $c_t^*$ in Figures \ref{fig:c_sens_group1} and \ref{fig:c_sens_group2} are increasing over time. An intuitive explanation is that, in the earlier stage, the agent tends to preserve wealth to maintain future investment opportunities. As time passes, the additional gain from preserving wealth for future investment gradually diminishes, and the agent is therefore more willing to allocate wealth to current consumption.
	\begin{figure}[htbp]
		\centering
		\includegraphics[width=0.90\textwidth]{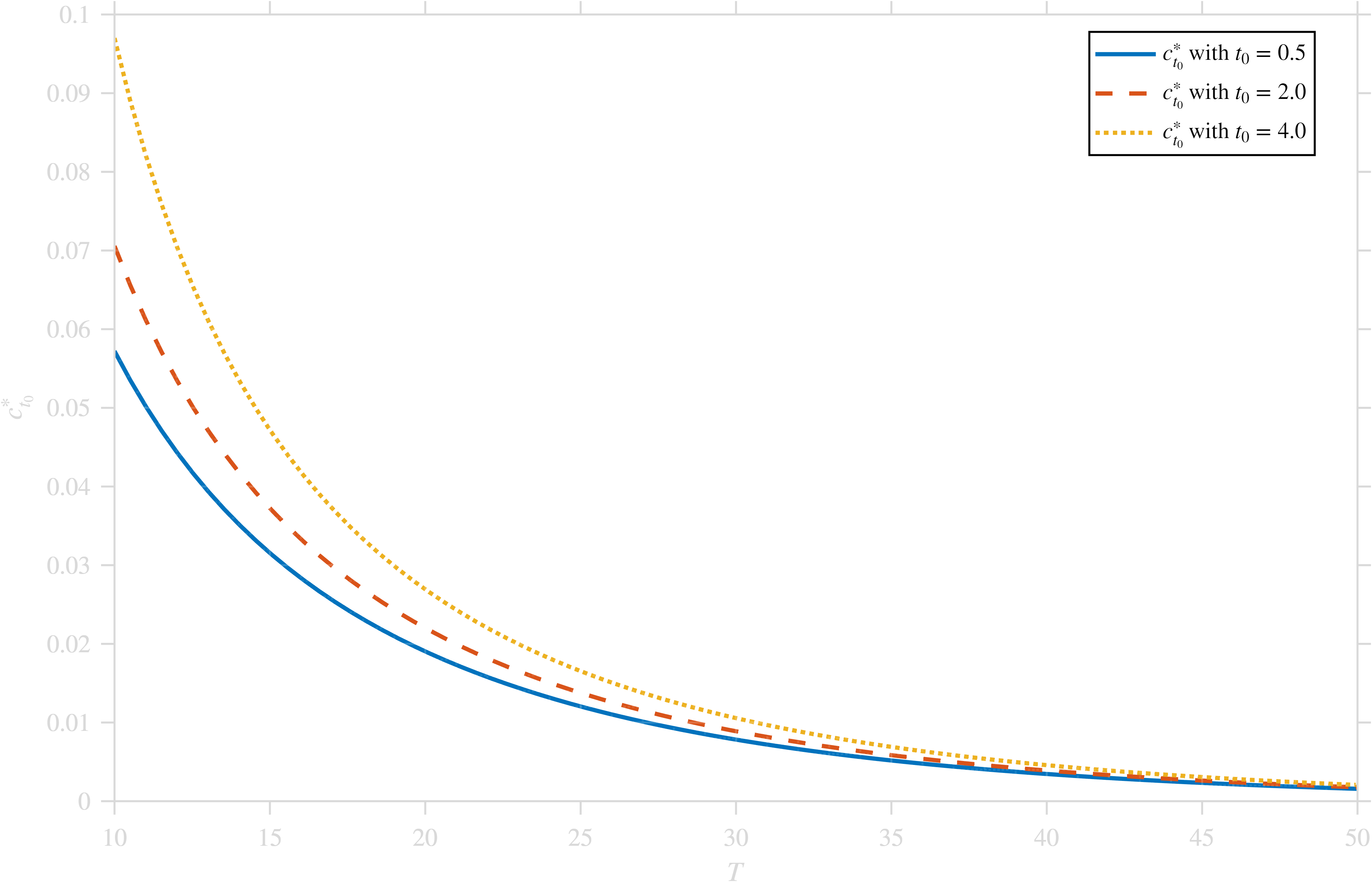}
		\caption{The effect of the terminal time $T$ on mean field equilibrium strategy $c_{t_0}^*$}
		\label{fig:long_horizon1}
	\end{figure}
	
	We now examine the effect of the investment horizon \(T\) on the MFE consumption strategy \(c_t^*\). As shown in Figure \ref{fig:long_horizon1}, for each fixed time \(t_0\), the consumption level \(c_{t_0}^*\) decreases as \(T\) increases. A longer planning horizon increases the continuation value of retained wealth, thereby raising the opportunity cost of current consumption. As a result, the agent has a stronger incentive to save and consume less in the present.
	
	Figure \ref{fig:long_horizon2} further shows that changes in $T$ affect only the trajectory of $c_t^*$, while leaving the terminal value $c_T^*$ unchanged. This property follows directly from the explicit expression for $c_t^*$. Indeed, by setting $t=T$, we obtain
	\begin{equation*}
		c_T^*=\varepsilon^{\frac{1}{\gamma-\theta\gamma-1}}.
	\end{equation*}
	Hence, $c^*_T$ depends only on the model parameters and is independent of $T$. Consequently, all curves in Figure \ref{fig:long_horizon2} attain the same terminal value.
	\begin{figure}[htbp]
		\centering
		\includegraphics[width=0.90\textwidth]{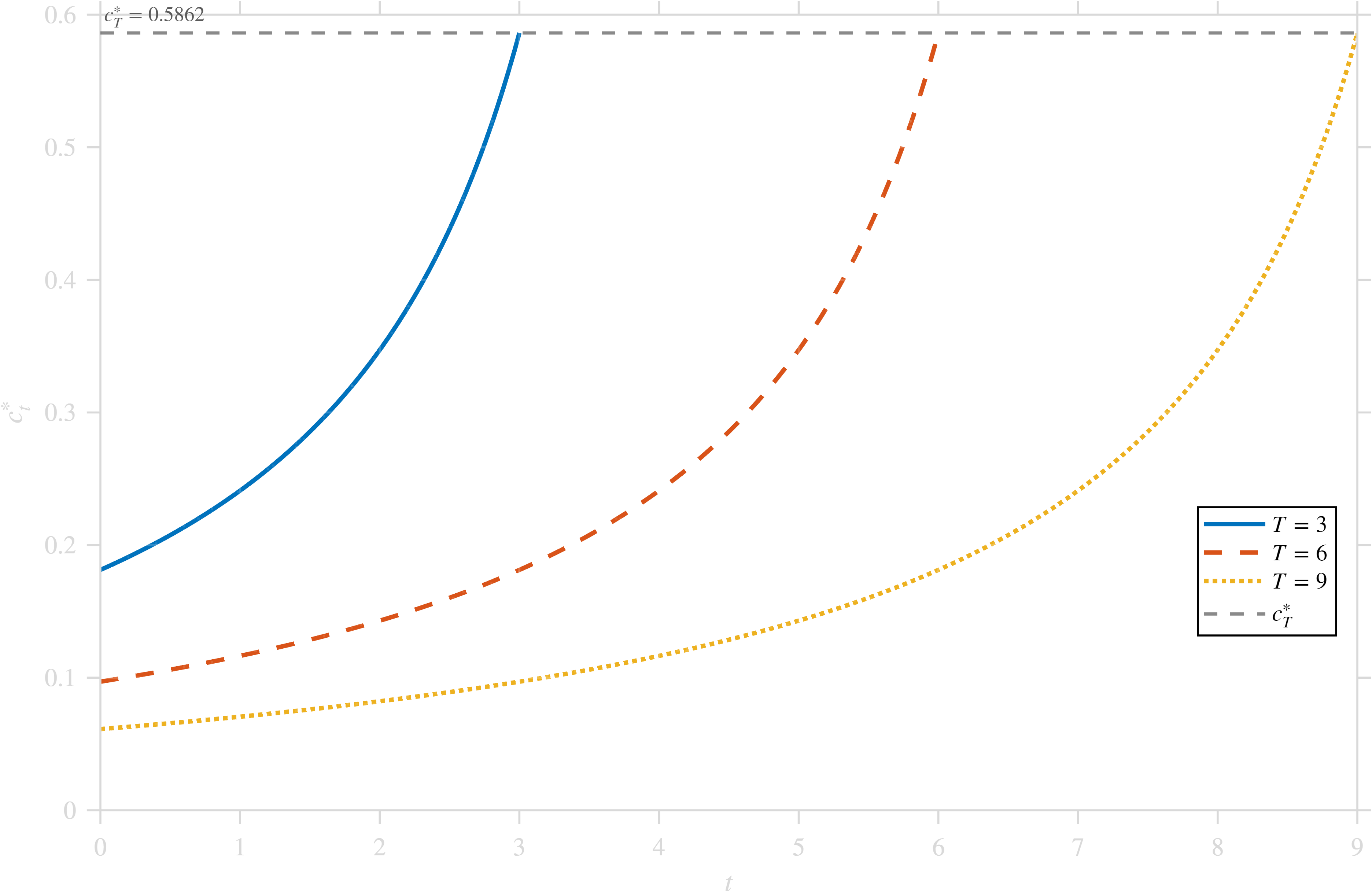}
		\caption{Independence of mean field equilibrium strategy $c^*_T$ from the terminal time $T$}
		\label{fig:long_horizon2}
	\end{figure}
	
	\section{Approximate Nash equilibrium in the N-player game}    
	Before constructing a $u_N$-Nash equilibrium for the $N$-player game, we introduce an auxiliary control problem based on the representative agent model, which plays a key role in the construction of approximate Nash equilibrium strategies. The auxiliary control problem is given by
	\begin{equation}\label{object_i*}
		\begin{aligned}
			&\sup_{(\pi^i,c^i)\in\mathcal{A}_i}\bar{J}_i((\pi^i,c^i);(m^*,\Gamma^*))\\            :=&\sup_{(\pi^i,c^i)\in\mathcal{A}_i}\mathbb{E}\left[\int_0^T U(c_t^iX_t^{i}(\Gamma^*_tm^*_t)^{-\theta_i};\gamma_i)dt+\varepsilon_iU(X_T^{i}(m^*_T)^{-\theta_i};\gamma_i)\right],
		\end{aligned}
	\end{equation}
	subject to
	\begin{equation*}
		\begin{cases}
			dX_t^{i} =\left(\pi_t^ib_i-c_t^i\right)X_t^{i} dt+\pi_t^i X_t^{i}\sigma_idW_t^i+\pi_t^i X_t^{i}\sigma_i^0dW_t^0-\pi_t^i X_{t-}^{i}dM_t^i, \\
			dm_t^* = m_t^* [\eta\left(\pi^*\right)-c_t^*]dt + \sigma^0 m_t^* \pi^* dW_t^0,
			\\d\Gamma_t^*=dc^*_t.
		\end{cases}
	\end{equation*}
	In this model, $m^*=(m_t^*)_{t\in[0,T]}$ and $\Gamma^*=(\Gamma_t^*)_{t\in[0,T]}$ are the fixed point obtained in the limiting case, and $\pi^*,c^*=(c_t^*)_{t\in[0,T]}$ are the deterministic MFE strategies. The following lemma provides the optimal strategy for the auxiliary control problem. Since the proof follows from the stochastic maximum principle and is similar to that of Theorem \ref{th3}, we omit the details.
	\begin{lem}\label{lmhat}
		Let $(\hat{\pi}^i,\hat{c}^i)\in\mathcal{A}_i$ be the optimal strategy of the auxiliary control problem, then for all $t\in[0,T]$, $\hat{\pi}^i_t=\hat{\pi}^i$ satisfies the following equation, 
		$$
		\begin{aligned}
			(\gamma_i-1)\left(\sigma_i^2+\left(\sigma_i^0\right)^2\right)\hat{\pi}^i-\theta_i\gamma_i\sigma_i^0 \sigma^0 \pi^* 
			-\lambda_i\left(\left(1-\hat{\pi}^i\right)^{\gamma_i-1}-1\right)+b_i=0
		\end{aligned}
		$$
		and 
		\[
		\hat c_t^i=
		\begin{cases}
			\dfrac{A_i B_i e^{A_i(t-T)}}{B_i(1-e^{A_i(t-T)})+A_i}, & A_i\ne0,\\
			\dfrac{B_i}{1+B_i(T-t)}, & A_i=0,
		\end{cases}
		\]
		where \(A_i\) and \(B_i\) are defined by
		$$
		\begin{aligned}
			A_i=\frac{1}{\gamma_i-1}[\hat{\rho}+\theta_i\gamma_i\frac{\rho}{\gamma-\theta\gamma-1}],B_i=\left(\varepsilon_i\varepsilon^{\frac{\theta_i\gamma_i}{\gamma-\theta\gamma-1}}\right )^{\frac{1}{\gamma_i-1}},
		\end{aligned}
		$$
		here
		\begin{equation*}
			\begin{aligned}
				\hat{\rho}=&-(\gamma_i-1)(b_i+\lambda_i)\hat{\pi}^i+\theta_i\gamma_i\eta(\pi^*)+\theta_i\gamma_i\left( \gamma_i-1 \right)\sigma^0_i\sigma^0\hat{\pi}^i\pi^*\\
				&-\frac{1}{2}\theta_i\gamma_i(\theta_i\gamma_i+1)\left(\sigma^0\right)^2\left(\pi^*\right)^2-\frac{1}{2}(\gamma_i-1)(\gamma_i-2)\left(\sigma_i^2+\left(\sigma_i^0\right)^2\right)\left(\hat{\pi}^i\right)^2\\
				&-\left[\left(1-\hat{\pi}^i\right)^{\gamma_i-1}-1\right]\lambda_i.
			\end{aligned}
		\end{equation*}
	\end{lem}
	
	By Bo et al. \cite[Lemma 4.2]{Bo2024a}, there exists a unique solution \(\hat{\pi}^i\in [D_0,1-\epsilon_0]\). We define this solution as $\hat{\pi}^i=\phi_i(\xi,\xi^i)$, which only depends on parameters $\xi$ and $\xi^i$. Analogously, we define $\hat{c}^i_t=\varphi_i(t,\xi,\xi^i)$.
	
	We then construct optimal strategies for the system with $N$ agents and verify that these strategies constitute an approximate Nash equilibrium. For $i=1,\cdots,N$, we recall the  objective functional \eqref{object_i} of agent $i$:
	\begin{equation}\label{object_i1}
		J_i((\pi^i,c^i),(\bm{\pi},\bm{c})^{-i})=\mathbb{E}\left[\int_0^T U(c_t^iX_t^i(\overline{c_tX_t})^{-\theta_i};\gamma_i)dt+\varepsilon_iU(X_T^i\overline{X_T}^{-\theta_i};\gamma_i)\right],
	\end{equation}
	where $(\bm{\pi},\bm{c})^{-i}:=\left((\pi^1,c^1),\ldots,(\pi^{i-1},c^{i-1}),(\pi^{i+1},c^{i+1}),\ldots,(\pi^N,c^N)\right)$. Subsequently, for $i=1,\cdots,N$, we consider a strategy of agent $i$
	\begin{equation}\label{pi*c*}
		\pi^{*,i}=\phi_i(\xi,\xi^i),\ c^{*,i}_t=\varphi_i(t,\xi,\xi^i),\qquad t\in[0,T].
	\end{equation}
	Then the SDE system for agent $i$ under strategy $\pi^{*,i}$ and $c^{*,i}_t$ is given by
	\begin{equation}\label{pic*}
		\begin{aligned}
			dX_t^{*,i} =\left(\pi^{*,i}b_i-c^{*,i}_t\right)X_t^{*,i} dt+\pi^{*,i} X_t^{*,i}\sigma_idW_t^i+\pi^{*,i} X_t^{*,i}\sigma_i^0dW_t^0-\pi^{*,i} X_{t-}^{*,i}dM_t^i.
		\end{aligned}
	\end{equation}
	
	To prove Theorem \ref{Nashequ}, we first establish the following lemma.
	\begin{lem}\label{mGamma-cX}
		Let Assumption \ref{Ao} hold. Then, for any $p\in\mathbb{R}$, there exist two constants $C_p^{*,1},C_p^{*,2}$ such that, for $t\in[0,T]$,
		\begin{equation}
			\mathbb{E}\left[\left(m^*_t\right)^p\right]\leqslant C_p^{*,1},\quad \mathbb{E}\left[\left(\Gamma^*_t\right)^p\right]\leqslant C_p^{*,2}.
		\end{equation}
	\end{lem}
	\begin{proof}
		We first discuss the boundedness of \(\pi^*\) and \(c_t^*\) for \(t\in[0,T]\).
		It follows from the proof of Theorem \ref{th3} that
		\(\pi^*\in(0,1-\epsilon_0]\), and that \(c^*\) is continuous and strictly
		positive on \([0,T]\). Hence, \(c^*\) is bounded away from zero and infinity
		on \([0,T]\). Therefore, for every \(p\in\mathbb R\),
		\[
		\mathbb{E}\left[\left(\Gamma_t^*\right)^p\right]=\left(c_t^*\right)^p\leqslant C_p^{*,2}.
		\]
		We next establish the boundedness of $\mathbb{E}\left[\left(m^*_t\right)^p\right]$. By \eqref{m^*},
		\begin{equation}\label{Em}
			\mathbb{E}\left[\left(m^*_t\right)^p\right]=\mathbb{E}\left[x_0^p \exp\left\{ p\int_0^t \left(\eta\left(\pi^* \right)-c_s^*-\frac{1}{2}\left(\sigma^0 \pi^*\right)^2\right)ds+p\int_0^t \sigma^0\pi^* dW_s^0 \right\}\right].
		\end{equation}
		Let
		\[
		\alpha_s=\eta\left(\pi^* \right)-c_s^*-\frac{1}{2}\left(\sigma^0 \pi^*\right)^2,
		\qquad
		\beta=\sigma^0\pi^*.
		\]
		By Assumption \ref{Ao}, there exist constants \(L_1\) and \(L_2\) such that,
		for all \(s\in[0,T]\),
		\begin{equation}\label{atbt}
			|\alpha_s|\leqslant L_1,\qquad|\beta|\leqslant L_2.
		\end{equation}
		Under the initial probability measure $\mathbb{P}$, consider the process defined by
		\begin{equation}
			Z_t=\exp\left\{p\int_{0}^{t}\beta dW_s^0-\frac{p^2}{2}\int_{0}^{t}\beta^2ds\right\}.
		\end{equation}
		From the Novikov condition, namely that $$\mathbb{E}\left[\exp\left\{\frac{1}{2}\int_{0}^{T}p^2\beta^2ds\right\}\right]\leqslant \exp\left\{\frac{1}{2}p^2L_2^2T\right\}<\infty,$$
		it follows that $Z_t$ is a martingale with $\mathbb{E}[Z_t] = 1$. Define a new probability measure \(\mathbb Q\sim\mathbb P\) by
		\begin{equation*}
			\frac{d\mathbb Q}{d\mathbb P}\bigg|_{\mathcal G_t}=Z_t .
		\end{equation*}
		Applying the change of measure, we have
		\begin{equation}\label{P-Q}
			\begin{aligned}
				\mathbb{E}\left[\exp\left\{p\int_{0}^{t}\beta dW_s^0\right\}\right]=\mathbb{E}^{\mathbb{Q}}\left[\exp\left\{\frac{p^2}{2}\int_{0}^{t}\beta^2ds\right\}\right]\leqslant \exp\left\{\frac{1}{2}p^2L_2^2T\right\}.
			\end{aligned}
		\end{equation}
		Substituting \eqref{atbt} and \eqref{P-Q} into \eqref{Em}, we obtain
		\begin{equation*}
			\mathbb{E}\left[\left(m^*_t\right)^p\right]\leqslant x_0^p \exp\left\{|p|L_1T+\frac{1}{2}p^2L_2^2T\right\}:=C_p^{*,1}.
		\end{equation*}
		This completes the proof of the lemma.
	\end{proof}
	\begin{lem}\label{EcX}
		Let Assumption \ref{Ao} hold. Then, for any $q\in\mathbb{R}$, there exist deterministic constants $C_q^{*,1}$ and $C^{*,2}$, independent of $N$ and $i$, such that, for every $i=1,\ldots,N$ and every admissible pair $(\pi^i,c^i)\in\mathcal A_i$,
		\begin{equation}
			\sup_{t\in[0,T]}
			\mathbb{E}\left[\left|c_t^{*,i}X_t^{*,i}\right|^q\right]
			\leqslant C_q^{*,1},
			\qquad
			\mathbb{E}\left[\int_0^T c_t^iX_t^i\,dt\right]
			\leqslant C^{*,2}.
		\end{equation}
	\end{lem}
	\begin{proof}
		We first establish the first inequality in the lemma. By Lemma \ref{lmhat}, \(c_t^{*,i}=\hat c_t^i\) is given explicitly in terms of
		\(A_i\) and \(B_i\). By Assumption \ref{Ao} and the explicit forms of
		\(A_i\) and \(B_i\), there exist constants  \(0<\bar A<\infty\)  and
		\(0<\underline B\leqslant \bar B<\infty\), independent of \(N\) and \(i\), such that
		\[
		|A_i|\leqslant \bar A,\qquad \underline B\leqslant B_i\leqslant \bar B .
		\]
		Hence, from the explicit formula of \(\hat c_t^i\), the family
		\(\{c_t^{*,i}\}_{i,N}\) is uniformly bounded away from zero and infinity on
		\([0,T]\). Consequently, for every \(q\in\mathbb R\), there exists a constant
		\(\tilde L_q^{*,1}\), independent of \(N\) and \(i\), such that
		\[
		\sup_{t\in[0,T]}\left|c_t^{*,i}\right|^q\leqslant \tilde L_q^{*,1},\qquad i=1,\ldots,N .
		\]
		Therefore, it remains to show that, for every \(q\in\mathbb R\), there exists a constant \(\tilde L_q^{*,2}\), independent of \(N\) and \(i\), such that
		\[
		\sup_{t\in[0,T]}\mathbb E\left[\left|X_t^{*,i}\right|^q\right]\leqslant \tilde L_q^{*,2},\qquad i=1,\ldots,N .
		\]
		Recalling that the wealth process $X^{*,i}_t$ satisfies \eqref{pic*}, 
		an application of It\^o's formula yields
		\begin{equation*}
			\begin{aligned}
				(X_t^{*,i})^q=&(x_0^i)^q\exp\left\{q\int_0^t\left[\pi^{*,i}(b_i+\lambda_i)-c_s^{*,i}-\frac12(\pi^{*,i})^2(\sigma_i^2+(\sigma_i^0)^2)\right]ds\right.\\
				&\left.+q\int_0^t\pi^{*,i}(\sigma_idW_s^i+\sigma_i^0dW_s^0)+q\int_0^t\ln(1-\pi^{*,i})\,dN_s^i\right\}.
			\end{aligned}
		\end{equation*}
		Similar to the proof of Lemma \ref{mGamma-cX}, we define the following process
		\begin{equation*}
			\begin{aligned}
				\widetilde Z_t=\exp&\left\{q\int_0^t\pi^{*,i}(\sigma_i dW_s^i+\sigma_i^0 dW_s^0)-\frac12q^2\int_0^t(\pi^{*,i})^2(\sigma_i^2+(\sigma_i^0)^2)ds\right.\\
				&\left.+q\int_{0}^{t}\ln(1-\pi^{*,i})dN_s^i+\int_{0}^{t}[1-(1-\pi^{*,i})^q]\lambda_ids\right\}.
			\end{aligned}
		\end{equation*}
		By the Novikov condition for L\'evy processes, we can verify that $(\widetilde Z_t)_{t\in[0,T]}$ is a martingale and that $\mathbb{E}[\widetilde Z_t] = 1$. Define a new probability measure $\widetilde{\mathbb Q}\sim\mathbb{P}$ by
		\[
		\frac{d\widetilde{\mathbb Q}}{d\mathbb{P}}\mid_{\mathcal{G}_t}= \widetilde Z_t,\quad t\in[0,T].
		\]
		Observe that
		\begin{equation*}
			\begin{aligned}
				\mathbb{E}[(X_t^{*,i})^q]
				&=(x_0^i)^q \mathbb{E}^{\widetilde{\mathbb Q}} \left[\exp \left\{ \int_0^t  \left[q\pi^{*,i}(b_i+\lambda_i)-qc_s^{*,i}\right.\right.\right.\\
				&\qquad\qquad\left.\left.\left. +\frac{1}{2}q(q-1)(\pi^{*,i})^2(\sigma_i^2+(\sigma_i^0)^2)+\lambda_i((1-\pi^{*,i})^q-1) \right]ds \right\} \right].\\
			\end{aligned}
		\end{equation*}
		Since \(\xi^i\), \(\pi^{*,i}\), and \(c_s^{*,i}\) are
		uniformly bounded, it follows that, for every \(q\in\mathbb R\), 
		\begin{equation}\label{supEX}
			\sup_{t\in[0,T]}\mathbb E\left[\left|X_t^{*,i}\right|^q\right]\leqslant \tilde L_q^{*,2},\qquad i=1,\ldots,N .
		\end{equation}
		Hence, we obtain
		\begin{equation*}
			\sup_{t\in[0,T]}\mathbb{E}\left[\left|c^{*,i}_tX^{*,i}_t\right|^q\right]\leqslant\sup_{t\in[0,T]}\left|c^{*,i}_t\right|^q\sup_{t\in[0,T]}\mathbb{E}\left[\left|X^{*,i}_t\right|^q\right]\leqslant C_q^{*,1}.
		\end{equation*}
		
		Next, we prove the second inequality. Let
		$(\pi^i,c^i)\in\mathcal A_i$ be arbitrary. Set
		\[
		A_t^i:=\int_0^t c_s^i\,ds,\qquad t\in[0,T].
		\]
		By Definition \ref{mathcalAi}, $A_T^i<\infty$, $\mathbb P$-a.s. The
		wealth process can be written as
		\[
		X_t^i=R_t^i e^{-A_t^i},
		\]
		where
		\[
		\begin{aligned}
			R_t^i
			=&x_0^i
			\exp\Bigg\{
			\int_0^t
			\left[
			\pi_s^i(b_i+\lambda_i)
			-\frac12(\pi_s^i)^2\left(\sigma_i^2+(\sigma_i^0)^2\right)
			\right]ds \\
			&\quad
			+\int_0^t\pi_s^i\left(\sigma_i dW_s^i+\sigma_i^0dW_s^0\right)
			+\int_0^t\ln(1-\pi_s^i)\,dN_s^i
			\Bigg\}.
		\end{aligned}
		\]
		Using the same argument as in the proof of the boundedness estimate \eqref{supEX}, we obtain, for any $\hat q>0$ and any $i=1,\ldots,N$,
		\begin{equation}\label{supER}
			\sup_{t\in[0,T]}\mathbb{E}\left[\left(X_t^{i}\right)^{\hat{q}}\right]
			\leqslant
			\sup_{t\in[0,T]}
			\mathbb{E}\left[\left(R_t^i\right)^{\hat{q}}\right]
			\leqslant \tilde{L}_{\hat{q}},
		\end{equation}
		where $\tilde L_{\hat q}$ is a deterministic constant independent of $N$, $i$, and the admissible pair $(\pi^i,c^i)$. Integrating the wealth equation \eqref{SDEi} and then taking expectations, we obtain
		\[
		\mathbb E\left[\int_0^T c_t^iX_t^i\,dt\right]
		=
		x_0^i
		+\mathbb E\left[\int_0^T \pi_t^ib_iX_t^i\,dt\right]
		-\mathbb E[X_T^i].
		\]
		Since $X_T^i\ge0$, $\pi^i$ is bounded, and $b_i$ is uniformly bounded under
		Assumption \ref{Ao}, it follows from \eqref{supER} with $\hat q=1$ that
		\[
		\begin{aligned}
			\mathbb E\left[\int_0^T c_t^iX_t^i\,dt\right]
			&\le
			x_0^i
			+K\sup_{t\in[0,T]}\mathbb E[X_t^i]  \\
			&\le
			C^{*,2}.
		\end{aligned}
		\]
		Here and throughout the rest of the paper, \(K\) denotes a deterministic constant independent of \(N\), \(i\), and the admissible pair \((\pi^i,c^i)\in\mathcal A_i\), whose value may change from line to line. This completes the proof of the lemma.
	\end{proof}
	\begin{rmk}
		By H\"older's inequality, Lemma \ref{EcX} yields the following consequence: for any $a\in(0,1)$ and any admissible control $(\pi^i,c^i)\in\mathcal A_i$, there exists a deterministic constant $C_a^*$, independent of $N$, $i$, and $(\pi^i,c^i)$, such that, for every $i=1,\ldots,N$,
		\[
		\begin{aligned}
			\mathbb E\left[\int_0^T (c_t^iX_t^i)^a\,dt\right]
			&\leqslant
			\left(
			\mathbb E\left[\int_0^T c_t^iX_t^i\,dt\right]
			\right)^a
			\left(
			\mathbb E\left[\int_0^T 1\,dt\right]
			\right)^{1-a}\leqslant C_a^* .
		\end{aligned}
		\]
	\end{rmk}
	\begin{lem}\label{EmEGamma}
		Let Assumption \ref{Ao} hold. Then, for any \(n>0\), the following assertions hold: for each \(t\in[0,T]\),
		\[
		\lim_{N\to\infty}
		\mathbb E\left[
		\left|
		\overline{X_t^*}-m_t^*
		\right|^n
		\right]=0,
		\]
		and
		\[
		\lim_{N\to\infty}
		\mathbb E\left[
		\int_0^T
		\left|
		m_t^*\Gamma_t^*
		-\overline{c_t^*}\,\overline{X_t^*}
		\right|^n\,dt
		\right]=0.
		\]
		Here, $\overline{X_t^{*}}=\left(\prod _{j=1}^N X^{*,j}_t\right)^{\frac{1}{N}},\overline{c_t^{*}}=\left(\prod _{j=1}^N c^{*,j}_t\right)^{\frac{1}{N}}$, and \(m_t^*,\Gamma_t^*\) are the limiting processes given in Theorem \ref{th3}.
	\end{lem}
	\begin{proof}
		For $i=1,\cdots,N$, set $Y_t^{*,i}:=\ln X_t^{*,i}$. Then
		\begin{equation*}
			\overline{Y_t^*}:=\frac1N\sum_{i=1}^N Y_t^{*,i}=\frac1N\sum_{i=1}^N\ln X_t^{*,i}=\ln \overline{X_t^*},
		\end{equation*}
		We first prove that, as $N\to\infty$,
		\[
		\mathbb{E}\left[\left|\overline{X_t^*}-m_t^*\right|^n\right]=\mathbb{E}\left[\left|e^{\overline{Y_t^*}}-e^{\ln m^*_t}\right|^n\right]\to0.
		\]
		By Lemmas \ref{mGamma-cX} and \ref{EcX}, for any $k>0$,
		\[
		\mathbb{E}\left[\left(\overline{X_t^*}\right)^k\right]\leqslant\prod_{i=1}^N\left\{\mathbb{E}\left[(X_t^{*,i})^k\right]\right\}^{\frac{1}{N}}\leqslant K,\qquad\mathbb{E}\left[(m_t^*)^k\right]\leqslant K.
		\]
		Hence, for any $\kappa>0$, 
		\[
		\sup_{N\geqslant1}\mathbb{E}\left[\left|\overline{X_t^*}-m_t^*\right|^{n(1+\kappa)}\right]\leqslant C_{n,\kappa}\sup_{N\geqslant 1}\mathbb{E}\left[|\overline{X_t^*}|^{n(1+\kappa)}+|m_t^*|^{n(1+\kappa)}\right]\leqslant K.
		\]
		where \(C_{n,\kappa}\) is a constant depending only on $n,\kappa$. Therefore, $\Big\{\big|\overline{X_t^*}-m_t^*\big|^n\Big\}_{N\geqslant1}$ is uniformly integrable. By Vitali's convergence theorem, it is enough to prove that
		\[
		e^{\overline{Y_t^*}}\to e^{\ln m^*_t}\qquad\text{in probability}.
		\]
		Since the exponential function is continuous, by the continuous mapping theorem it suffices to show that
		\[
		\overline{Y_t^*}\to \ln m_t^*\qquad\text{in probability}.
		\]
		Applying It\^o's formula to $Y_t^{*,i}=\ln X_t^{*,i}$, we obtain
		\begin{equation*}
			\begin{aligned}
				Y_t^{*,i}=&\ln x_0^i+\int_0^t\left[\pi^{*,i}(b_i+\lambda_i)-c_s^{*,i}-\frac12(\pi^{*,i})^2\left(\sigma_i^2+(\sigma_i^0)^2\right)+\lambda_i\ln(1-\pi^{*,i})\right]ds\\
				&+\int_0^t \pi^{*,i}\sigma_idW_s^i+\int_0^t \pi^{*,i}\sigma_i^0dW_s^0+\int_0^t \ln(1-\pi^{*,i})dM_s^i.
			\end{aligned}
		\end{equation*}
		On the other hand, by \eqref{m^*} in Theorem \ref{th3},
		\begin{equation*}
			\begin{aligned}
				\ln m_t^*=&\ln x_0+\int_0^t\left[\pi^*(b+\lambda)-c_s^*-\frac12(\pi^*)^2\left(\sigma^2+(\sigma^0)^2\right)+\lambda\ln(1-\pi^*)\right]ds\\
				&+\int_0^t \pi^*\sigma^0dW_s^0.
			\end{aligned}
		\end{equation*}
		Define
		\begin{equation*}
			\begin{aligned}
				\overline{\mu_t^*}:=\frac1N\sum_{i=1}^N\mathbb{E}[Y_t^{*,i}\mid\mathcal F_t^0].
			\end{aligned}
		\end{equation*}
		Then
		\begin{equation}\label{Ylnm}
			\overline{Y_t^*}-\ln m_t^*=\left(\overline{Y_t^*}-\overline{\mu_t^*}\right)+\left(\overline{\mu_t^*}-\ln m_t^*\right).
		\end{equation}
		We first estimate the first term. From the independence of $(W^1,\dots,W^N)$ and $(M^1,\dots,M^N)$, we have
		\begin{equation*}
			\begin{aligned}
				\mathbb{E}\left[\left|\overline{Y_t^*}-\overline{\mu_t^*}\right|^2\right]&=\frac1{N^2}\mathbb{E}\left[\left|\sum_{i=1}^N \left(Y_t^{*,i}-\mathbb{E}[Y_t^{*,i}\mid\mathcal F_t^0]\right)\right|^2\right]\\
				&=\frac1{N^2}\sum_{i=1}^N \mathbb{E}\left[\left(\int_0^t\pi^{*,i}\sigma_idW_s^i\right)^2+\left(\int_0^t\ln(1-\pi^{*,i})dM_s^i\right)^2\right.\\
				&\qquad\qquad\,\,\left.+2\left(\int_0^t\pi^{*,i}\sigma_idW_s^i\right)\left(\int_0^t \ln(1-\pi^{*,i})dM_s^i\right)\right].
			\end{aligned}
		\end{equation*}
		Using It\^o's isometry and the isometry for compensated Poisson integrals, we get
		\begin{equation*}
			\begin{aligned}
				\mathbb{E}\left[\left|\overline{Y_t^*}-\overline{\mu_t^*}\right|^2\right]\leqslant\frac{2}{N^2}\sum_{i=1}^N\mathbb{E}\left\{\int_0^t\left[(\pi^{*,i})^2\sigma_i^2+\lambda_i(\ln(1-\pi^{*,i}))^2
				\right]ds\right\}\leqslant\frac{K}{N}.
			\end{aligned}
		\end{equation*}
		Consequently,
		\begin{equation}\label{Ymu}
			\mathbb{E}\left[\left|\overline{Y_t^*}-\overline{\mu_t^*}\right|^2\right]\to0\qquad\text{as }N\to\infty.	
		\end{equation}
		
		We next analyze the second term in \eqref{Ylnm} . Write
		\[
		\overline{\mu_t^*}-\ln m_t^*
		=
		A^N+\int_0^t B_s^Nds+\int_0^t C_s^NdW_s^0,
		\]
		where
		\begin{equation*}
			\left\{
			\begin{aligned}
				&A^N:=\frac1N\sum_{i=1}^N \ln x_0^i-\ln x_0,\\
				&B_s^N:=\frac1N\sum_{i=1}^N\left[\pi^{*,i}(b_i+\lambda_i)-c_s^{*,i}-\frac12(\pi^{*,i})^2\left(\sigma_i^2+(\sigma_i^0)^2\right)+\lambda_i\ln(1-\pi^{*,i})\right]\\
				&\qquad-\left[\pi^*(b+\lambda)-c_s^*-\frac12(\pi^*)^2\left(\sigma^2+(\sigma^0)^2\right)+\lambda\ln(1-\pi^*)\right],\\
				&C_s^N:=\frac1N\sum_{i=1}^N\pi^{*,i}\sigma_i^0-\pi^*\sigma^0.
			\end{aligned}
			\right.
		\end{equation*}
		By Assumption \ref{Ao}, for any bounded continuous function $f$, we have
		\begin{equation}\label{ftof}
			\begin{aligned}
				\frac{1}{N}\sum_{i=1}^{N}f(\xi^i)=\int_{\mathcal{O}} f(\zeta)d\nu_0^N\to\int_{\mathcal{O}} f(\zeta)d\nu_0=f(\xi).
			\end{aligned}
		\end{equation}
		For $\zeta=(x_\zeta,\lambda_\zeta,b_\zeta,\sigma_\zeta,\sigma^0_\zeta,\varepsilon_\zeta,\gamma_\zeta,\theta_\zeta) \in \mathcal{O}$, $\ln x_{\zeta}$ is a bounded continuous function. 
		Therefore, by \eqref{ftof}, we have
		\[
		\frac1N\sum_{i=1}^N \ln x_0^i\to \ln x_0
		\qquad \text{as }N\to\infty,
		\]
		that is,
		\[
		A^N\to0
		\qquad \text{as }N\to\infty.
		\]
		We now consider $C_s^N$, fix $s\in[0,T]$, let $\Phi(\zeta)$ denote the solution on $[D_0,1-\epsilon_0]$ of the equation defining as following:
		\begin{equation*}
			G(x,\zeta)=(\gamma_\zeta-1)\left(\sigma_\zeta^2+\left(\sigma_\zeta^0\right)^2\right)x-\theta_\zeta\gamma_\zeta\sigma_\zeta^0 \sigma^0 \pi^*
			-\lambda_\zeta\left(\left(1-x\right)^{\gamma_\zeta-1}-1\right)+b_\zeta=0.
		\end{equation*}
		Then, by uniqueness of the solution,
		\[
		\pi^{*,i}=\Phi(\xi^i),
		\qquad
		\pi^*=\Phi(\xi).
		\]
		We claim that $\Phi(\zeta)$ is continuous in $\zeta$. Indeed, $G(x,\zeta)$ is continuous in $(x,\zeta)$. Let $\{\zeta_n\}_{n\ge1}\subset\mathcal O$ satisfy $\zeta_n\to\zeta$, and define $x_n:=\Phi(\zeta_n)$. Since $\{x_n\}_{n\ge1}\subset [D_0,1-\epsilon_0]$ is bounded, there exists a subsequence \(\{x_{n_k}\}_{k\ge1}\) such that
		\[
		x_{n_k}\to \bar x\in [D_0,1-\epsilon_0].
		\]
		Using $G(x_{n_k},\zeta_{n_k})=0$ and the continuity of $G$, we obtain $G(\bar x,\zeta)=0$. By uniqueness, $\bar x=\Phi(\zeta)$. Hence every convergent subsequence of $\{x_n\}_{n\ge1}$ has the same limit $\Phi(\zeta)$, and therefore
		\[
		\Phi(\zeta_n)\to \Phi(\zeta).
		\]
		Thus, $\Phi$ is continuous on $\mathcal O$. Clearly, $\Phi(\zeta)\sigma_\zeta^0$ is a bounded continuous function. Therefore, by \eqref{ftof}, as $N \to \infty$,
		\[
		\frac1N\sum_{i=1}^N \pi^{*,i}\sigma_i^0=\frac1N\sum_{i=1}^N\Phi(\xi^i)\sigma_i^0\to \Phi(\xi)\sigma^0=\pi^*\sigma^0.
		\]
		That is,
		\[
		C_s^N\to0\qquad \text{as }N\to\infty.
		\]
		Next we prove that
		\[
		B_s^N\to0\qquad\text{as }N\to\infty.
		\]
		For fixed $s\in[0,T]$, denote $\Psi_s(\zeta)$ as follows: 
		\begin{equation*}
			\Psi_s(\zeta)
			=\left\{
			\begin{aligned}
				&\frac{A(\zeta)B(\zeta)e^{A(\zeta)(s-T)}}{B(\zeta)\big(1-e^{A(\zeta)(s-T)}\big)+A(\zeta)}\qquad A(\zeta)\neq0,\\
				&\frac{B(\zeta)}{1+B(\zeta)(T-s)}\qquad\qquad\qquad\quad\, A(\zeta)=0,
			\end{aligned}
			\right.
		\end{equation*}
		where
		\begin{equation*}
			A(\zeta)=\frac{1}{\gamma_\zeta-1}[\hat{\rho}(\zeta)+\theta_\zeta\gamma_\zeta\frac{\rho}{\gamma-\theta\gamma-1}],B(\zeta)=\left(\varepsilon_\zeta\varepsilon^{\frac{\theta_\zeta\gamma_\zeta}{\gamma-\theta\gamma-1}}\right )^{\frac{1}{\gamma_\zeta-1}},
		\end{equation*}
		here $\hat{\rho}(\zeta)$ is specified in Lemma \ref{lmhat} and defined by
		\begin{equation*}
			\begin{aligned}
				&\hat{\rho}(\zeta)=-(\gamma_\zeta-1)(b_\zeta+\lambda_\zeta)\Phi(\zeta)+\theta_\zeta\gamma_\zeta\eta(\pi^*)+\theta_\zeta\gamma_\zeta(\gamma_\zeta-1)\sigma_\zeta^0\sigma^0 \Phi(\zeta)\pi^*\\
				&-\frac12\theta_\zeta\gamma_\zeta(\theta_\zeta\gamma_\zeta+1)(\sigma^0)^2(\pi^*)^2
				-\frac12(\gamma_\zeta-1)(\gamma_\zeta-2)\left(\sigma_\zeta^2+(\sigma_\zeta^0)^2\right)\Phi^2(\zeta)
				-\lambda_\zeta\left((1-\Phi(\zeta))^{\gamma_\zeta-1}-1\right).
			\end{aligned}
		\end{equation*}
		It follows from the uniqueness of the solution that
		\[
		c_s^{*,i}=\Psi_s(\xi^i),\qquad c_s^*=\Psi_s(\xi).
		\]
		By the boundedness and continuity of $\Phi(\zeta)$, we conclude that $\hat \rho(\zeta)$, $A(\zeta)$, and $B(\zeta)$ are bounded and continuous  on \(\mathcal O\). Moreover, the
		formula defining \(\Psi_s(\zeta)\) is well defined and strictly positive, and
		the case \(A(\zeta)=0\) is the continuous extension of the case
		\(A(\zeta)\ne0\). Therefore, \(\Psi_s(\zeta)\) is bounded and continuous on \(\mathcal O\). Define
		\[
		f_s(\zeta):=\Phi(\zeta)(b_\zeta+\lambda_\zeta)-\Psi_s(\zeta)-\frac12\Phi^2(\zeta)\left(\sigma_\zeta^2+(\sigma_\zeta^0)^2\right)+\lambda_\zeta\ln\left(1-\Phi(\zeta)\right).
		\]
		Then $f_s$ is bounded and continuous. Hence, by \eqref{ftof},
		\begin{equation*}
			\begin{aligned}
				\frac1N\sum_{i=1}^Nf_s(\xi^i)\to f_s(\xi).
			\end{aligned}
		\end{equation*}
		Therefore,
		\[
		B_s^N\to0,\qquad \text{as }N\to\infty.
		\]
		Using the elementary inequality, the Cauchy--Schwarz inequality, and It\^o's isometry, we obtain
		\begin{equation*}
			\begin{aligned}
				\mathbb{E}\left[\left|\overline{\mu_t^*}-\ln m_t^*\right|^2\right]
				&\le 3|A^N|^2+3\mathbb{E}\left[\left|\int_0^t B_s^N\,ds\right|^2\right]
				+3\mathbb{E}\left[\left|\int_0^t C_s^N\,dW_s^0\right|^2\right] \\
				&\le 3|A^N|^2+3t\int_0^t |B_s^N|^2\,ds+3\int_0^t |C_s^N|^2\,ds.
			\end{aligned}
		\end{equation*}
		Therefore, by the dominated convergence theorem, for each fixed \(t\in[0,T]\),
		\begin{equation}\label{mulnm}
			\mathbb{E}\left[\left|\overline{\mu_t^*}-\ln m_t^*\right|^2\right]\to 0
			\qquad \text{as }N\to\infty.
		\end{equation}
		By \eqref{Ymu} and \eqref{mulnm}, together with the form of \eqref{Ylnm}, we arrive at
		\begin{equation*}\label{EYlnm}
			\mathbb{E}\left[\left|\overline{Y_t^*}-\ln m_t^*\right|^2\right]\leqslant2\mathbb{E}\left[\left|\overline{Y_t^*}-\overline{\mu_t^*}\right|^2\right]+2\mathbb{E}\left[\left|\overline{\mu_t^*}-\ln m_t^*\right|^2\right]\to0\qquad \text{as }N\to\infty.
		\end{equation*}
		Hence
		\begin{equation}\label{Exm0}
			\mathbb{E}\left[\left|\overline{X_t^*}-m_t^*\right|^n\right]\to0\qquad\text{as }N\to\infty.
		\end{equation}
		
		We now prove that, for each fixed $t\in[0,T]$,
		\[
		|\overline{c_t^*}-\Gamma_t^*|^n\to0
		\qquad\text{as } N\to\infty.
		\]
		Since \(c_t^{*,i}=\Psi_t(\xi^i)\), \(c_t^*=\Gamma_t^*=\Psi_t(\xi)\), and
		\(\ln\Psi_t\) is bounded and continuous on \(\mathcal O\), Assumption
		\ref{Ao} yields
		\[
		\ln\overline{c_t^*}
		=
		\frac1N\sum_{i=1}^N\ln\Psi_t(\xi^i)
		\to
		\ln\Psi_t(\xi)
		=
		\ln\Gamma_t^* .
		\]
		By the continuity of the exponential function, it follows that
		\(\overline{c_t^*}\to\Gamma_t^*\). Since
		\(h(x)=|x-\Gamma_t^*|^n\) is continuous, we further obtain
		\begin{equation}\label{ctoGamma}
			|\overline{c_t^*}-\Gamma_t^*|^n
			=
			h(\overline{c_t^*})
			\to
			h(\Gamma_t^*)
			=
			0
			\qquad \text{as } N\to\infty.
		\end{equation}
		
		Finally, for any \(n>0\), we have
		\[
		\begin{aligned}
			&\mathbb{E}\left[
			\left|
			\overline{c_t^*}\overline{X_t^*}
			-\Gamma_t^*m_t^*
			\right|^n
			\right]\leqslant
			C_n(\overline{c_t^*})^n
			\mathbb{E}\left[
			\left|
			\overline{X_t^*}-m_t^*
			\right|^n
			\right]
			+
			C_n
			|\overline{c_t^*}-\Gamma_t^*|^n
			\mathbb{E}\left[(m_t^*)^n\right],
		\end{aligned}
		\]
		where \(C_n\) is a constant depending only on \(n\). By Lemmas \ref{mGamma-cX} and \ref{EcX}, together with \eqref{Exm0} and \eqref{ctoGamma}, we conclude that
		\[
		\lim_{N\to\infty}\mathbb{E}\left[\left|m_t^*\Gamma_t^*-\overline{c_t^*}\overline{X_t^*}\right|^n\right]=0.
		\]
		Moreover, Lemmas \ref{mGamma-cX} and \ref{EcX} imply that
		\[
		\sup_{N\ge1}\sup_{t\in[0,T]}\mathbb E\left[
		\left|
		m_t^*\Gamma_t^*
		-\overline{c_t^*}\,\overline{X_t^*}
		\right|^n
		\right]\leqslant K.
		\]
		Hence, by Tonelli's theorem and the dominated convergence theorem,
		\[
		\lim_{N\to\infty}\mathbb E\left[
		\int_0^T
		\left|
		m_t^*\Gamma_t^*
		-\overline{c_t^*}\,\overline{X_t^*}
		\right|^ndt
		\right]
		=
		\int_0^T \lim_{N\to\infty}\mathbb{E}\left[\left|m_t^*\Gamma_t^*-\overline{c_t^*}\overline{X_t^*}\right|^n\right]\,dt
		=0.
		\]
		This completes the proof.
	\end{proof}
	\begin{proof}[Proof of Theorem~\ref{Nashequ}]
		For any $(\pi^i,c^i)\in\mathcal{A}_i$ and $(\pi^{*,i},c^{*,i})$ defined by \eqref{pi*c*}, we first introduce three objective functionals:
		\begin{equation*}
			\begin{cases}
				J_i((\pi^i,c^i),(\bm{\pi}^*,\bm{c}^*)^{-i})=\mathbb{E}\left[\int_0^T U(c_t^iX_t^i(\overline{c_t^{*,-i}}\overline{X_t^{*,-i}})^{-\theta_i};\gamma_i)dt+\varepsilon_iU(X_T^i(\overline{X_T^{*,-i}})^{-\theta_i};\gamma_i)\right]\\
				J_i((\pi^{*,i},c^{*,i}),(\bm{\pi}^*,\bm{c}^*)^{-i})=\mathbb{E}\left[\int_0^T U(c_t^{*,i}X_t^{*,i}(\overline{c_t^{*}}\overline{X_t^{*}})^{-\theta_i};\gamma_i)dt+\varepsilon_iU(X_T^{*,i}(\overline{X_T^{*}})^{-\theta_i};\gamma_i)\right]\\
				\bar{J}_i((\pi^i,c^i);(m^*,\Gamma^*))=\mathbb{E}\left[\int_0^T U(c_t^iX_t^i(m^*_t\Gamma^*_t)^{-\theta_i};\gamma_i)dt+\varepsilon_iU(X_T^i(m^*_T)^{-\theta_i};\gamma_i)\right],
			\end{cases}
		\end{equation*}
		Here, for $i=1,\cdots,N$ and $t\in[0,T]$,
		\begin{equation*}
			\begin{aligned}
				\overline{X_t^{*,-i}}=\left(X^i_t\prod_{\substack{1\le j\le N\\ j\ne i}} X^{*,j}_t\right)^{\frac{1}{N}},\qquad\overline{c_t^{*,-i}}=\left(c^i_t\prod_{\substack{1\le j\le N\\ j\ne i}} c^{*,j}_t\right)^{\frac{1}{N}}.
			\end{aligned}
		\end{equation*}
		Our goal is to prove \eqref{supJ} in Theorem \ref{Nashequ}. To this end, note that
		\begin{equation}\label{SUPJ}
			\begin{aligned}
				&\quad\sup\limits_{(\pi^i,c^i)\in\mathcal{A}_i} J_i\left( (\pi^i,c^i),(\bm{\pi}^{*},\bm{c}^{*})^{-i} \right)-J_i\left( (\pi^{*,i},c^{*,i}),(\bm{\pi}^{*},\bm{c}^{*})^{-i}\right)\\
				&\leqslant\sup\limits_{(\pi^i,c^i)\in\mathcal{A}_i} \left(J_i\left( (\pi^i,c^i),(\bm{\pi}^{*},\bm{c}^{*})^{-i} \right)-\bar{J}_i\left( (\pi^i,c^i);(m^*,\Gamma^*) \right)\right)\\
				&\quad+\sup\limits_{(\pi^i,c^i)\in\mathcal{A}_i} \bar{J}_i\left( (\pi^i,c^i);(m^*,\Gamma^*) \right)-J_i\left( (\pi^{*,i},c^{*,i}),(\bm{\pi}^{*},\bm{c}^{*})^{-i}\right)\\
			\end{aligned}
		\end{equation}
		For the first term of RHS of \eqref{SUPJ}, we have
		\begin{equation*}
			\begin{aligned}
				&\quad J_i\left( (\pi^i,c^i),(\bm{\pi}^{*},\bm{c}^{*})^{-i} \right)-\bar{J}_i\left( (\pi^i,c^i),(m^*,\Gamma^*) \right)\\	
				&=\mathbb{E}\left\{\int_0^T \left[U(c_t^iX_t^i(\overline{c_t^{*,-i}}\overline{X_t^{*,-i}})^{-\theta_i};\gamma_i)-U(c_t^iX_t^i(\overline{c_t^{*}}\overline{X_t^{*}})^{-\theta_i};\gamma_i)\right]dt\right\}\\
				&\quad+\mathbb{E}\left\{\int_0^T \left[U(c_t^iX_t^i(\overline{c_t^{*}}\overline{X_t^{*}})^{-\theta_i};\gamma_i)-U(c_t^iX_t^i(m^*_t\Gamma^*_t)^{-\theta_i};\gamma_i)\right]dt\right\}\\
				&\quad+\mathbb{E}\left[\varepsilon_iU(X_T^i(\overline{X_T^{*,-i}})^{-\theta_i};\gamma_i)-\varepsilon_iU(X_T^i(\overline{X_T^{*}})^{-\theta_i};\gamma_i)\right]\\
				&\quad+\mathbb{E}\left[\varepsilon_iU(X_T^i(\overline{X_T^{*}})^{-\theta_i};\gamma_i)-\varepsilon_iU(X_T^i(m^*_T)^{-\theta_i};\gamma_i)\right]\\
				&=:I_1^i+I_2^i+I_3^i+I_4^i.
			\end{aligned}
		\end{equation*}
		We now turn to the analysis of $I_1$, by the definition of the utility function,
		\begin{equation*}
			\begin{aligned}
				I_1^i=\frac{1}{\gamma_i}\,\mathbb{E}\left\{\int_0^T
				(c_t^iX_t^i)^{\gamma_i}(\overline{c_t^{*}} \overline{X_t^{*}} )^{-\theta_i \gamma_i}\left[\left(
				\frac{c_t^iX_t^i}{c_t^{*,i}X_t^{*,i}}\right)^{-\frac{\theta_i \gamma_i}{N}}-1\right]dt\right\}.
			\end{aligned}
		\end{equation*}
		Choose \(\delta>0\) such that $0<2\delta<\min\{\underline\gamma,1-\overline\gamma\}$, and let
		\[
		N_0:=\left\lceil \frac{\overline\gamma}{\delta}\right\rceil,
		\]
		For \(x>0\), using \( |e^z-1|\le |z|e^{|z|} \), we have
		\[
		\left|x^{-\frac{\theta_i\gamma_i}{N}}-1\right|
		=
		\left|e^{-\frac{\theta_i\gamma_i}{N}\ln x}-1\right|
		\le
		\frac{\theta_i\gamma_i}{N}|\ln x|\,e^{\frac{\theta_i\gamma_i}{N}|\ln x|}.
		\]
		Moreover, for any \(N\ge N_0\),
		\[
		e^{\frac{\theta_i\gamma_i}{N}|\ln x|}
		\le e^{\delta|\ln x|}
		\le x^{\delta}+x^{-\delta}.
		\]
		Together with \( |\ln x|\le K(x^{\delta}+x^{-\delta}) \), this yields
		\begin{equation}\label{x^-1}
			\left|x^{-\frac{\theta_i\gamma_i}{N}}-1\right|
			\le \frac{K}{N}\bigl(x^{2\delta}+x^{-2\delta}\bigr),
			\qquad x>0,\; N\ge N_0.
		\end{equation}
		Substituting $x=\frac{c_t^iX_t^i}{c_t^{*,i}X_t^{*,i}}$ into \eqref{x^-1}, we obtain that, for $N\geqslant N_0$,
		\begin{align*}
			I_1^i&\leqslant\frac{K}{N}\mathbb{E}\left[\int_0^T
			(\overline{c_t^{*}} \overline{X_t^{*}} )^{-\theta_i \gamma_i}\left((c_t^iX_t^i)^{\gamma_i+2\delta}(c_t^{*,i}X_t^{*,i})^{-2\delta}\right.\right.\\
			&\left.\left.\qquad\qquad+(c_t^iX_t^i)^{\gamma_i-2\delta}(c_t^{*,i}X_t^{*,i})^{2\delta}\right)dt\right]\\
			&:=\frac{K}{N}(I_{1,1}^i+I_{1,2}^i).
		\end{align*}
		We next estimate $I_{1,1}^i$, choose $r,s,\eta>1$ such that
		\[
		\frac1{r}+\frac1{s}+\frac1{\eta}=1,\qquad \eta(\overline\gamma+2\delta)<1 .
		\]
		Then applying H\"older's inequality, we obtain
		\begin{align*}
			I_{1,1}^i&\leqslant\left\{\mathbb{E}\left[\int_0^T(\overline{c_t^{*}}\overline{X_t^{*}})^{-\theta_i\gamma_i r}dt\right]\right\}^{\frac{1}{r}}\left\{\mathbb{E}\left[\int_0^T(c_t^{*,i}X_t^{*,i})^{-2\delta s}dt\right]\right\}^{\frac{1}{s}}\\
			&\quad\times\left\{
			\mathbb{E}\left[\int_0^T
			(c_t^iX_t^i)^{(\gamma_i+2\delta)\eta}dt\right]
			\right\}^{\frac{1}{\eta}}.
		\end{align*}
		It follows from Lemma \ref{EcX} that \(I_{1,1}^i\le K\). Similarly, \(I_{1,2}^i\le K\). Hence, for \(N\ge N_0\),	we have
		$$I_1^i\le \frac{K}{N}.$$ 
		Consequently,
		\begin{equation*}
			I_1^i\to0 \qquad\text{as } N\to\infty.
		\end{equation*}  
		By the same argument, we also have
		\begin{equation*}
			I_3^i\to0 \qquad\text{as } N\to\infty.
		\end{equation*}
		
		We next estimate $I_2^i$, recall that
		\begin{equation*}
			I_2^i=\frac{1}{\gamma_i}\mathbb E\!\left[\int_0^T(c_t^iX_t^i)^{\gamma_i}\left((\overline{c_t^*}\overline{X_t^*})^{-\theta_i \gamma_i}-(m_t^*\Gamma_t^*)^{-\theta_i \gamma_i}\right)\,dt
			\right].
		\end{equation*}
		By the mean value theorem, we obtain
		\begin{equation*}
			\begin{aligned}
				I_2^i
				&\leqslant
				\theta_i
				\mathbb E\left\{
				\int_0^T\left[
				(c_t^iX_t^i)^{\gamma_i}
				|\overline{c_t^*}\overline{X_t^*}-m_t^*\Gamma_t^*|
				\left((\overline{c_t^*}\overline{X_t^*})^{-\theta_i \gamma_i-1}+(m_t^*\Gamma_t^*)^{-\theta_i \gamma_i-1}\right)\right]dt
				\right\}\\
				&=:\theta_i(I_{2,1}^i+I_{2,2}^i).
			\end{aligned}
		\end{equation*}
		Choose \(\hat r,\hat s,\hat\eta>1\) satisfying
		\[
		\frac1{\hat r}+\frac1{\hat s}+\frac1{\hat\eta}=1,
		\qquad
		\overline\gamma\hat r<1.
		\]
		We have, from H\"older's inequality, that
		\begin{equation*}
			\begin{aligned}
				I_{2,1}^i
				&\leqslant\left\{\prod_{j=1}^{N}\left[\int_{0}^{T}\mathbb{E}(c^{*,j}_tX^{*,j}_t)^{-(\theta_i\gamma_i +1)\hat{s}}dt\right]^{\frac{1}{N}}\right\}^{\frac{1}{\hat{s}}}\left\{\mathbb{E}\left[\int_0^T(c_t^{i}X_t^{i})^{\gamma_i\hat{r}}dt\right]\right\}^{\frac{1}{\hat{r}}}\\
				&\quad\times\left\{
				\mathbb{E}\left[\int_0^T
				\left|\overline{c_t^*}\overline{X_t^*}-m_t^*\Gamma_t^*\right|^{\hat{\eta}}dt\right]
				\right\}^{\frac{1}{\hat{\eta}}}.
			\end{aligned}
		\end{equation*}
		Therefore, by Lemmas \ref{EcX} and \ref{EmEGamma}, we have $I_{2,1}^i\to0$ as $N\to\infty$. The same argument implies $I_{2,2}^i\to0$ as $N\to\infty$. Hence,
		\[
		I_2^i\to0\qquad \text{as }N\to\infty.
		\]
		By the same argument, we further obtain
		\[
		I_4^i\to0, \qquad \text{as }N\to\infty.
		\]
		
		For the second term of RHS of \eqref{SUPJ}, we can derive that
		\begin{equation*}
			\begin{aligned}
				&\quad\sup\limits_{(\pi^i,c^i)\in\mathcal{A}_i} \bar{J}_i\left( (\pi^i,c^i);(m^*,\Gamma^*) \right)-J_i\left( (\pi^{*,i},c^{*,i}),(\bm{\pi}^{*},\bm{c}^{*})^{-i}\right)\\
				&=\sup\limits_{(\pi^i,c^i)\in\mathcal{A}_i} \bar{J}_i\left( (\pi^i,c^i);(m^*,\Gamma^*) \right)-\bar{J}_i\left( (\pi^{*,i},c^{*,i});(m^*,\Gamma^*)\right)\\
				&\quad+\bar{J}_i\left( (\pi^{*,i},c^{*,i});(m^*,\Gamma^*)\right)-J_i\left( (\pi^{*,i},c^{*,i}),(\bm{\pi}^{*},\bm{c}^{*})^{-i}\right).
			\end{aligned}
		\end{equation*}
		In view of Lemma \ref{lmhat} and the strategy of agent $i$ given in \eqref{pi*c*}, the first term above yields
		\begin{equation*}
			\begin{aligned}
				&\sup\limits_{(\pi^i,c^i)\in\mathcal{A}_i} \bar{J}_i\left( (\pi^i,c^i);(m^*,\Gamma^*) \right)-\bar{J}_i\left( (\pi^{*,i},c^{*,i});(m^*,\Gamma^*)\right)\\
				&=\bar{J}_i\left( (\hat{\pi}^i,\hat{c}^i);(m^*,\Gamma^*) \right)-\bar{J}_i\left( (\pi^{*,i},c^{*,i});(m^*,\Gamma^*)\right)=0.
			\end{aligned}
		\end{equation*}
		Following a similar argument as in the proof of convergence of $I_2^i+I_4^i$, we have
		\begin{equation*}
			\begin{aligned}
				&\bar{J}_i\left( (\pi^{*,i},c^{*,i});(m^*,\Gamma^*)\right)-J_i\left( (\pi^{*,i},c^{*,i}),(\bm{\pi}^{*},\bm{c}^{*})^{-i}\right)\to0\qquad \text{as }N\to\infty.
			\end{aligned}
		\end{equation*}
		Combining the estimates for \(I_1^i,I_2^i,I_3^i,I_4^i\) with the preceding analysis, we obtain the desired result. Thus, we complete the proof.
	\end{proof}

	
\end{document}